%% file: OddRamseyMultipartite.tex
\documentclass{article}
\usepackage[margin=1.2in]{geometry} 
\usepackage{graphicx} 

\pdfoutput = 1

\usepackage{amsmath,amsthm,amssymb,amsfonts, fancyhdr, comment, parskip, tikz, caption, subcaption, graphicx, xcolor}
\usepackage[colorlinks]{hyperref}
\usetikzlibrary{graphs}
\usetikzlibrary{shadings}
\usetikzlibrary{shapes.geometric}
\usetikzlibrary{arrows}
\usepackage{mathtools}
\usepackage{standalone}
\usepackage{upgreek}
\usepackage{bm}
\usepackage{tikz-cd}
\usepackage{enumerate}
\usepackage{enumitem}
\usepackage{soul}

\hypersetup{
    colorlinks=true,
    linkcolor=blue,
    citecolor=green,
    filecolor=magenta,
    urlcolor=blue,
    pdfborder={0 0 0}, 
}

\usepackage{mdframed}
\usepackage[capitalise]{cleveref}

\newtheorem{theorem}{Theorem}

\newtheorem{lemma}[theorem]{Lemma}

\newtheorem{claim}[theorem]{Claim}

\theoremstyle{definition}
\newtheorem{definition}[theorem]{Definition}

\newcommand\cH{{\mathcal H}}
\newcommand\cC{{\mathcal C}}
\newcommand{\cD}{\mathcal{D}}

\newcommand\cM{{\mathcal M}}
\newcommand\cK{{\mathcal K}}
\newcommand\cF{{\mathcal F}}

\newcommand{\oram}{r_{odd}}

\definecolor{green}{RGB}{50,205,50}

\title{Odd Ramsey numbers of  multipartite graphs and hypergraphs}

\author{Nicholas Crawford \footnote{University of Colorado Denver, \texttt{nicholas.2.crawford@ucdenver.edu}.}
\and
Emily Heath \footnote {California State Polytechnic University Pomona, \texttt{eheath@cpp.edu}.}
\and Owen Henderschedt \footnote {Auburn University, \texttt{olh0011@auburn.edu}.}
\and Coy Schwieder \footnote {Iowa State University, \texttt{cschwi@iastate.edu}.}
\and Shira Zerbib \footnote {Iowa State University, \texttt{zerbib@iastate.edu}. Supported by NSF CAREER award no. 2336239  and Simons Foundation award no. MP-TSM-00002629.}}
\date{\today}

\begin{document}

\maketitle

\begin{abstract}
Given a hypergraph $G$ and a subhypergraph $H$ of $G$, the \emph{odd Ramsey number} $\oram(G,H)$ is the minimum number of colors needed to edge-color $G$ so that every copy of $H$ intersects some color class in an odd number of edges.
Generalizing a result of \cite{BHZ} in two different ways, in this paper we prove $\oram \left(K_{n,n}, K_{2,t} \right)=\frac{n}{t} + o(n)$ for all $t\geq 2$,  and $\oram \left(\cK^{(k)}_{n,\dots,n}, \cK_{1,\dots,1,2,2} \right) = \frac{n}{2} + o(n)$ for all $k\geq 2$.  The latter is the first result studying odd Ramsey numbers for hypergraphs. 
\end{abstract}

\section{Introduction}

The {\em generalized Ramsey number} $f(K_n,K_p,q)$, introduced by Erd\H{o}s and Shelah  \cite{Erdos-Shelah1}, is the minimum number of colors needed to edge-color the complete graph on $n$ vertices so that every copy of $K_p$ receives at least $q$ colors.   
In the special case where $q=2$, determining the generalized Ramsey numbers is equivalent to finding the diagonal Ramsey numbers $R(p,p)$. The problem can be further extended by replacing $K_n$ and $K_p$ by a hypergraph  $G$ and a subhypergraph $H$, respectively. Generalized Ramsey numbers have been studied extensively, see e.g.,~\cite{AFM,BEHK,BCDP,BED,CH1,CH2,CFLS,er-gy, Mubayi1,mubayi2004}. 

Recently, a variant of this problem has been introduced (see \cite{alon-codes,BDLP,BHZ,Versteegen}), asking for edge-colorings of $G$ in which every copy $C$ of $H$ intersects some color class in an odd number of edges. In this case, we say that $C$ {\em has an odd color class}. The minimum number of colors needed to produce such a coloring is called the \emph{odd Ramsey number}, $\oram(G,H)$. 

The study of $\oram(K_n,H)$ was initially motivated by a work of Alon~\cite{alon-codes} on bounding the size of \emph{graph-codes}. 
Given a family $\mathcal{H}$ of graphs on vertex set $[n]$, an \emph{$\mathcal{H}$-(graph)-code} is a collection of graphs $\mathcal{F}$ on vertex set $[n]$ such that the symmetric difference of any two members in $\cF$ is not a graph in $\mathcal{H}$. We denote the maximum possible cardinality of an $\mathcal{H}$-code by $D_{\mathcal{H}}(n)$, and  the maximum possible density of an $\mathcal{H}$-code by $$d_{\mathcal{H}}(n) = \displaystyle\frac{D_{\mathcal{H}}(n)}{2^{\binom{n}{2}}}.$$ In \cite{Versteegen}, Versteegen showed that $d_{H}(n) \geq \oram(K_n,H)^{-(e(H)/2 + 1)}$.  Thus, an upper bound on $\oram(K_n,H)$ implies a lower bound for $d_H(n)$.    
In this context, Versteegen \cite{Versteegen} obtained lower bounds on $\oram(K_n,H)$, when $H$ is a $K_4$-free graph with an even number of edges. To this end, he used a modification of the explicit coloring in \cite{CH1}.
In \cite{BDLPodd}, Boyadzhiyska, Das, Lesgourgues, and Petrova established a connection between odd Ramsey numbers and results from coding theory.  They showed $\oram(K_n, K_{t,n-t}) \in \{n-1, n\}$, and $$ (1 + o(1)) \left( \frac{n}{t} \right)^{1/\lceil s/2 \rceil} \leq \oram \left(K_n, K_{s,t} \right) \leq O \left(n^{\frac{2s + 2t - 4}{st}} \right),$$ where the upper bound is a direct corollary of the bound on $f \left(K_n, K_{s,t}, \frac{st}{2} + 1 \right)$ from  \cite{AFM}. 

In \cite{alon-codes} it was observed that a coloring in \cite{CH1} modifying constructions in \cite{CFLS,Mubayi1}  implies $\oram(K_{n},K_4)\le n^{o(1)}$.  
Bennett, Heath, and Zerbib \cite{BHZ}, and independently Ge, Xu, and Zhang \cite{GXZ}, then proved $\oram(K_{n},K_5)\le n^{o(1)}$. The next non-trivial case, $\oram(K_n, K_8) \leq n^{o(1)}$, was proven recently by Yip \cite{Yip}. All these results were obtained by providing an explicit (deterministic) edge-coloring of $K_n$.

For host graphs other than $K_n$, not much is known yet. In \cite{BHZ} it was shown that $\oram(K_{n,n},K_{2,2})=\frac{1}{2}n+o(n)$. The proof used a probabilistic coloring technique that utilized the {\em conflict-free hypergraph matching method} (see below).
 Joos and Mubayi \cite{JM} used this technique to obtain an upper bound on the generalized Ramsey number $f(K_n, C_4,3)$. 
 A similar approach has been used by various authors to obtain upper bounds on other generalized Ramsey numbers,  
 see e.g., \cite{BBHZ, BHZ, LM, GHPSZ}.

In this paper, we generalize the result $\oram(K_{n,n},K_{2,2})=\frac{1}{2}n+o(n)$ from \cite{BHZ} in two different ways.

\begin{theorem} \label{thm: main_graph}
    For any integer $t\geq 2$, we have 
    \[\oram \left(K_{n,n}, K_{2,t} \right)=\frac{n}{t} + o(n).\]
\end{theorem}

\begin{theorem} \label{thm: main_hypergraph}
    For any integer $k\geq 2$, we have 
    \[\oram \left(\cK^{(k)}_{n,\dots,n}, \cK_{1,\dots,1,2,2} \right) = \frac{n}{2} + o(n).\]
\end{theorem}

The proofs of the upper bounds in both theorems use the {\em Tripartite Matching Theorem}, a new version of the conflict-free hypergraph matching method recently introduced by Joos, Mubayi, and Smith \cite{JMS}. Theorem \ref{thm: main_hypergraph} is the first result studying odd Ramsey numbers for hypergraphs.

The organization of the paper is as follows. In Section \ref{sec:method}, we outline the Tripartite Matching Theorem \cite{JMS}. The proof of Theorem \ref{thm: main_graph} is then given in Section \ref{sec:pmain}, and the proof of Theorem \ref{thm: main_hypergraph} is given in Section \ref{sec:hyp}.

\section{The Tripartite Matching Theorem} \label{sec:method}

The conflict-free hypergraph matching method 
is a powerful probabilistic tool that, given some conditions, guarantees the existence of a matching avoiding certain negative constraints (conflicts) in a hypergraph. 
The method was introduced by Glock, Joos, Kim, K\"uhn, and Lichev  \cite{GJKKL} and independently by Delcourt and Postle \cite{DP}, and it has already seen many applications, including bounds for odd Ramsey and generalized Ramsey numbers (see e.g., \cite{BBHZ,BDLP,BHZ,GHPSZ,JM,JMS,LM}).

In this paper, we use a new version of the method, the Tripartite Matching Theorem, which was proven recently by Joos, Mubayi, and Smith \cite{JMS} (Theorem \ref{thm:blackbox} below). Very recently Bal and Bennett improved the upper bound on $f(K_n, C_4,3)$ using the Tripartite Matching Theorem \cite{bb-25}. The key difference between the original method and the new version involves properties of the matching obtained by the theorem. In the original method, there is a hypergraph $\cH$ and a {\em conflict system} $\cC$ for $\cH$, where $V(\cC) = E(\cH)$. The edges of $\cC$ correspond to \emph{conflicts}, namely sets of edges in $\cH$ that should not be included together in the resulting matching. As long as some conditions on the degrees in $\cH$ and $\cC$ hold, the method guarantees the existence of an \emph{almost perfect matching} of $\cH$ that does not contain any of the conflicts in $\cC$. 

Applying this method to edge-coloring, one defines a hypergraph $\cH$ in which an almost perfect matching corresponds to an edge-coloring of a dense subgraph of $G$ with $k$ colors. The conflicts in $\cC$ are defined to correspond to such colorings that would not satisfy the odd- or generalized- Ramsey requirements. Thus, if the conditions of the theorem hold, one obtains a matching corresponding to a coloring of most of the edges in $G$. The edges that remain uncolored must be colored in a second step, usually by a random coloring with a new set of colors of size $o(k)$; then it must be shown that this can be done in a way that preserves the odd- or generalized-Ramsey properties required. 

In the Tripartite Matching Theorem, the method was modified to produce a matching that corresponds to a coloring of {\em all} the edges in $G$ in one step. 
  This was done by introducing a second conflict system $\cD$ in addition to $\cC$. In this version,  a matching in $\cH$ contains two types of edges: the first type corresponds to a coloring of the majority of the edges in $G$ much like the initial version, and the second type corresponds to a coloring of the rest of the edges. The conflict system $\cC$ consists of conflicts containing only edges of the first type, while the new conflict system $\cD$ consists of conflicts containing both types of edges. This replaces the need to have a second coloring step.

We now state the Tripartite Matching Theorem.
Given a hypergraph $\cH$, and a vertex $v \in V(\cH)$, the  \emph{degree} $\deg_\cH(v)$ of $v$ is the number of edges in $\cH$ containing $v$.  
The maximum degree and minimum degree of $\cH$ are denoted by $\Delta(\cH)$ and $\delta(\cH)$, respectively.  For $j \geq 2$, $\Delta_j(\cH)$ denotes the maximum $j$-codegree, that is, the maximum number of edges in $\cH$ containing a particular set of $j$ vertices, over all vertex subsets of size $j$. Given a subset of vertices $X\subseteq V(\mathcal{H})$, $\delta_X(\mathcal{H})$ and $\Delta_X(\mathcal{H})$ denote the minimum and maximum degree in $\mathcal{H}$ of any vertex from $X$, respectively. An \textit{$X$-perfect matching} in $\mathcal{H}$ is a matching which saturates every vertex in $X$.

For a hypergraph $\cC$ and an integer $k$, let $\cC^{(k)}$ be the set of edges in $\cC$ of size $k$.  For a vertex $u \in V(\cC)$, let $\cC_u$ denote the hypergraph $\{C \setminus \{u\} \colon C \in E(\cC), u \in C\}$.  Given a hypergraph $\cH$, a hypergraph $\cC$ is a \emph{conflict system} for $\cH$ if $V(\cC) = E(\cH)$.  A set of edges $E \subset \cH$ is \emph{$\cC$-free} if $E$ contains no subset $C \in \cC$.

Assume $\ell \geq 3$, $d > 0$, $p\ge 1$, $q\ge 0$, and $r \geq 1$ are integers, and $\varepsilon \in (0,1)$.  Let $P, Q, R$ be disjoint sets such that $d^\varepsilon \leq |P| \leq |P \cup Q| \leq \exp(d^{\varepsilon^3})$.  Let $\cH_1$ be a $(p+q)$-uniform hypergraph  on vertex set $P\cup Q$ whose every edge consists of $p$ vertices from $P$ and $q$ vertices from $Q$. Let $\cH_2$ be a $(1+r)$-uniform hypergraph on vertex set $P\cup R$ whose edges consist of one vertex from $P$ and $r$ vertices from $R$.  Finally, let $\cC$ be a conflict system for $\cH_1$ and $\cD$ be a conflict system for $\cH := \cH_1 \cup \cH_2$.  

Assume $\cH$ satisfies the following degree conditions:
\begin{enumerate}[label=(H\arabic*)]
    \item \label{H1} $(1 - d^{-\varepsilon}) d \leq \delta_P(\cH_1) \leq \Delta(\cH_1) \leq d$ (if $\cH_1$ satisfies this condition we  say that $\cH_1$ is {\em essentially $d$-regular)};
    \item \label{H2} $\Delta_2(\cH_1) \leq d^{1-\varepsilon}$;
    \item \label{H3} $\Delta_R(\cH_2) \leq d^{\varepsilon^4} \delta_P(\cH_2)$;
    \item \label{H4} $d(x,v) \leq d^{-\varepsilon} \delta_P(\cH_2)$ for each $x \in P$ and $v \in R$.
\end{enumerate}

If $\cC$ satisfies the following  conditions, we say $\cC$ is \emph{$(d, \ell, \varepsilon)$-bounded}:
\begin{enumerate}[label=(C\arabic*)]
    \item \label{C1} $3 \leq |C| \leq \ell$ for all $C \in \cC$;
    \item \label{C2} $\Delta(\cC^{(j)}) \leq \ell d^{j-1}$ for all $3 \leq j \leq \ell$;
    \item \label{C3} $\Delta_{j'}(\cC^{(j)}) \leq d^{j-j'-\varepsilon}$ for all $3 \leq j \leq \ell$ and $2 \leq j' \leq j-1$.   
\end{enumerate}

Let $\cD^{(j_1, j_2)}$ denote the set of conflicts in $\cD$ consisting of $j_1$ edges from $\cH_1$ and $j_2$ edges from $\cH_2$.  Similarly, let $\Delta_{j'_1, j'_2}(\cD)$ denote the maximum degree among sets of edges $F = F_1 \cup F_2 \subseteq \cH$ consisting of $|F_1| = j'_1$ edges from $\cH_1$ and $|F_2| = j'_2$ edges from $\cH_2$.  Lastly, for $x \in P$, let $\cD_x$ denote the set of conflicts in $\cD$ containing $x$ in their $\cH_2$-part, and $\cD_{x,y}$ denote the set of conflicts containing both $x$ and $y$ in their $\cH_2$-part.

If $\cD$ satisfies the following  conditions for all $x,y \in P$ and $0 \leq j_1 \leq \ell, 2 \leq j_2 \leq \ell$, we say that $\cD$ is \emph{$(d, \ell, \varepsilon)$-simply-bounded}:
\begin{enumerate}[label=(D\arabic*)]
    \item \label{D1} $2 \leq |D \cap \cH_2| \leq |D| \leq \ell$ for each conflict $D \in \cD$;
    \item \label{D2} $|\cD_x^{(j_1,j_2)}| \leq d^{j_1 + \varepsilon^4} \delta_P(\cH_2)^{j_2}$; 
    \item \label{D3} $\Delta_{j',0}(\cD_x^{(j_1,j_2)}) \leq d^{j_1 - j' - \varepsilon} \delta_P(\cH_2)^{j_2}$ for each $0 \leq j' \leq j_1$;
    \item \label{D4} $|\cD_{x,y}^{(j_1,j_2)}| \leq d^{j_1 - \varepsilon} \delta_P(\cH_2)^{j_2}$.
\end{enumerate}

Assuming that $\mathcal{C}$ is $(d,\ell,\varepsilon)$-bounded and $\mathcal{D}$ is $(d,\ell,\varepsilon)$-simply-bounded, Joos et al. proved the following theorem.

\begin{theorem}[\cite{JMS}] \label{thm:blackbox}
    For $p + q = k \geq 2$, there exists $\varepsilon_0 > 0$ such that for all $\varepsilon \in (0, \varepsilon_0)$, there exists $d_0$ such that given the above setup, the following holds for all $d \geq d_0$: there exists a $P$-perfect matching $\cM \subseteq \cH$ which contains none of the conflicts from $\cC \cup \cD$.  Furthermore, at most $d^{-\varepsilon^4} |P|$ vertices of $P$ belong to an edge in $\cH_2 \cap \cM$.  
\end{theorem}

\section{Proof of Theorem \ref{thm: main_graph}} \label{sec:pmain}

We start by showing that 
$\oram \left(K_{n,n}, K_{2,t} \right)>\frac{n}{t}$. Indeed, let the partite sets of $K_{n,n}$ be $X$ and $Y$, and suppose the edges are colored by at most   $\frac{n}{t}$ colors. By the Cauchy–Schwarz inequality, every vertex $x\in X$ has at least $\frac{n}{t}{t\choose 2}$  pairs of vertices $u,v \in Y$ so that $xu,xv$ are colored by the same color. It follows that there are  at least $\frac{n^2}{t}{t\choose 2}> {n\choose 2}(t-1)$ pairs  $u,v \in Y$ so that there exists some $x\in X$ with the edges $xu,xv$ having the same color. Since there are ${n\choose 2}$ pairs of vertices in $Y$, by the pigeonhole principle, there must exist at least $t$ vertices $x_1,\dots,x_t \in X$ and a pair vertices $u,v \in Y$, so that $x_iu, x_iv$ has the same color $i$ for all $i\in [t]$. But this gives a copy of $K_{t,2}$ that has no odd color class.

Let $N_1$ and $N_2$ be disjoint sets of $\lceil \frac{n}{t}\rceil$ and $\lceil n^\delta \rceil$ colors, respectively,  where $\delta<1$ will be determined later in the proof.
To prove $\oram \left(K_{n,n}, K_{2,t} \right)\leq \frac{n}{t} + o(n)$, it suffices to find an edge-coloring of $K_{n,n}$ with colors taken from the set $N_1  \cup N_2$, such that each copy of $K_{2,t}$ has an odd color class.

Let $K_{n,n}$ have partite sets $X$ and $Y$.  Let $U = \binom{X \cup Y}{2}$ and $V = \cup_{i \in N_1} V_i$, where each $V_i$ is a copy of $X \cup Y$.  We denote the copy of $v \in X \cup Y$ in $V_i$ by $v^i$.  Let $\cH_1$ be the $(2(t+1)^2)$-uniform hypergraph with vertex set $U \cup V$, whose edges are defined as follows: for every $i \in N_1$,  for every copy $S'$ of $K_{t+1,t+1}$ in $K_{n,n}$, and for every perfect matching $M$ in $S'$, let $S=S'-M$, and define an edge  $e_{S,i}$ of $\cH_1$  as

    $$
        e_{S,i} =\{v^i \colon v\in V(S)\}\cup \{xy\in E(S) \} \cup \left\{xx'\in\binom{V(S)\cap X}{2} \right\} \cup \left\{ yy'\in \binom{V(S)\cap Y}{2}\right\}.
    $$
Note that each $e_{S,i} \in \cH_1$ corresponds to a copy $S$  of the graph $K_{t+1,t+1} - (t+1)K_2 $ in $k_{n,n}$, whose all edges  are colored by the color $i \in N_1$.

Let $\cH_2$ be the 3-uniform hypergraph with vertex set $E(K_{n,n}) \cup W$, where $W = \cup_{i \in N_2} W_i$ and each $W_i$ is a copy of $X \cup Y$.  We denote the copy of $w \in X \cup Y$ in $W_i$ as $w^i$.  
For each $xy \in E(K_{n,n})$ and each $i \in N_2$, let  $\cH_2$ contain the edge
$$
    e_{xy,i} = \left\{ xy, x^{i}, y^{i} \right\}.
$$

Note that every edge $e_{xy,i}$ of $\cH_2$ corresponds to an edge $xy \in K_{n,n}$ that is colored by the color $i \in N_2$. 
Further, a matching in $\cH_1$ corresponds to a set of edge-disjoint monochromatic copies of $K_{t+1, t+1}-(t+1)K_2$ in $K_{n,n}$,  such that no vertex of $K_{n,n}$ appears in two copies of $K_{t+1, t+1} - M$ of the same color. A matching in $\cH_2$ corresponds to a set of edges in $K_{n,n}$ that are properly colored. A matching in $\cH := \cH_1 \cup \cH_2$ is the union of two such matchings, and it corresponds to a well-defined (partial) edge-coloring of $K_{n,n}$ (that is, no edge gets more than one color). 

To align our notation with those of Theorem \ref{thm:blackbox}, define $P := U$, $Q := V$, and $R := W$.  For the rest of the proof, we refer to the edges of $\cH$ as \emph{tiles}, to the edges of $K_{n,n}$ as {\em edges}, and the edges of the complement of $K_{n,n}$ as {\em non-edges}. Note that every tile of $cH_1$ contains both edges and non-edges.  

Let $d = \frac{n^{2t+1}}{t!}$, $\varepsilon\in (0,\frac{1}{2t+2})$, and $\delta\in (1-(2t+1)\varepsilon^4,1)$.

\begin{claim}
Conditions \ref{H1}-\ref{H4} hold,
    and $d^\varepsilon \leq |P| \leq |P \cup Q| \leq \exp(d^{\varepsilon^3})$ for sufficiently large $n$.
\end{claim}
\begin{proof}
We first show that the hypergraph $\cH_1$ is essentially $d$-regular for $d = \frac{n^{2t+1}}{t!}$. To see this we show that for any vertex $u \in V(\cH_1)$, we have $$(1-d^{-\varepsilon})d = d - d^{1-\varepsilon} \leq d- O \left(n^{2t+1 - \varepsilon (2t+1)} \right) \leq d- O \left(n^{2t+1 - 1} \right) = d - O(n^{2t}) \leq \deg_{\cH_1} (u) \leq d.$$  First, suppose $u \in V_i$ for some $i\in N_1$.  There are $\frac{n^{2t+1}}{t!} - O(n^{2t-1})$ many tiles in $\cH_1$ containing $u$ since there are $\binom{n}{t+1}$ ways to choose the vertices in the vertex part not containing $u$, $\binom{n-1}{t}$ ways to choose the vertices in the same vertex part as $u$, and $(t+1)!$ ways to choose which perfect matching to remove. If instead $u = xy$ for $x \in X, y \in Y$, then there are $\binom{n-1}{t}$ ways to choose the remaining vertices in each vertex part, there are $t \cdot t!$ ways to choose a matching to remove (we know $xy$ is not in the matching), and $\lceil \frac{n}{t}\rceil$ ways to choose the color. Similarly, if $u = xx'$ for $x,x' \in X$ (or $yy'$ for $y,y' \in Y$), there are $\binom{n-2}{t-1}$ ways to choose the remaining vertices of the same vertex part, $\binom{n}{t+1}$ ways to choose the vertices of the other vertex part, $(t+1)!$ ways to choose the matching, and $\lceil \frac{n}{t}\rceil$ ways to choose the color. Thus, in each case, we have $\deg_{\cH_1} (u) = \frac{n^{2t+1}}{t!} - O(n^{2t})$, satisfying \ref{H1}.

    Furthermore, we have $\Delta_2 (\cH_1) \leq d^{1- \varepsilon}$ for all $\varepsilon \in (0,\frac{1}{2t+2})$, since each pair of vertices $u,v \in V(\cH_1)$ is contained in at most $O(n^{2t})$ tiles in $\cH_1$.  To see this, given $u,v \in V(\cH_1)$, then either at least two vertices and a color of the copy of $K_{t+1,t+1}$ are given, or at least three of the vertices of a copy of $K_{t+1,t+1}$ are given.  Hence, there are at most $O(n^{2t})$ tiles containing both $u,v$, satisfying \ref{H2}.

    Now consider some $w^i\in R$. Then $\deg_R(w^i) = n$ since $w$ has exactly $n$ neighbors in $K_{n,n}$. This implies that for any $\varepsilon \in (0,\frac{1}{2t+2})$ and $\delta \in (1-(2t+1)\varepsilon^4,1)$ we have $n \leq n^{(2t+1)\varepsilon^4}\cdot n^{\delta}$. Thus, $\Delta_R(\mathcal{H}_2)\leq d^{\varepsilon^4}\delta_P(\mathcal{H}_2)$ and condition \ref{H3} is satisfied.

 To check \ref{H4}, fix $v\in R$, say $v = w^i$, where $i\in N_2$. Now fix $x\in P$ where $x = z_1z_2$ for some $z_1,z_2\in X\cup Y$. If $z_1z_2\notin E(K_{n,n})$ then it cannot appear in a hyperedge with $w^i$. Similarly, if $w\notin \{z_1,z_2\}$ then $z_1z_2$ also cannot appear in a hyperedge with $w^i$. Finally if without loss of generality, $w = z_1$, then $wz_2$ is in exactly one edge, namely $\{wz_2, w^i,z_2^i\}$, with $w^i$. Thus, $d(x,v) \leq 1$. Note that $d^{-\varepsilon} n^\delta$ is always bigger than 1 for $\varepsilon \in (0, \frac{1}{2t+2})$ and sufficiently large $\delta < 1$ since $\varepsilon (2t+1) < \delta$. Thus \ref{H4} is satisfied.
 
   Finally, note that 
    $$
        d^\varepsilon 
        \leq n^{\frac{2t+1}{2t+2}} \leq O(n^2) 
        = |P|
        \leq |P \cup Q|
        = O(n^2)
        \leq \exp \left(\frac{n^{\frac{2t+1}{(2t+2)^3}}}{t!} \right)
        = \exp \left( d^{\varepsilon^3} \right)
    $$
    for sufficiently large $n$.
\end{proof}

\subsection{The conflict system $\cC$}\label{sec:conC}

We say that an edge-colored copy $K$ of $K_{2,r}$ is {\em bad} if it has no odd color class. 
Recall that every tile $T$ in $\cH_1$ corresponds to a copy $S_T$ of the graph $K_{t+1,t+1}-(t+1)K_2$, whose edges are colored with some color $i$. If   $C = \{ T_1, \dots, T_j\}$ is a matching in $\cH_1$, denote by $E_C$ the colored edges in $\bigcup_{i=1}^j S_{T_i}$. If $K$ is a bad copy of $K_{2,r}$ for some $r\ge 1$,  we say that $K$ is {\em reducible} if there exist $1 \leq r_1 \leq r_2 \leq \dots \leq r_k < r$, with $k\ge 2$, so that  the edges of $K$ can be partitioned into bad  copies $K_1$ of $K_{2,r_1}$, $K_2$ of $K_{2,r_2}$, $\dots$, $K_k$ of $K_{2,r_k}$, all having the same vertex side of size 2 as $K$. Otherwise, we say that $K$ is {\em irreducible}. (See an example of a reducible bad copy of $K_{2,6}$ in \cref{fig:reducible}.) 
Note that a bad copy of $K_{2,1}$ is irreducible.

\begin{figure}[h!]
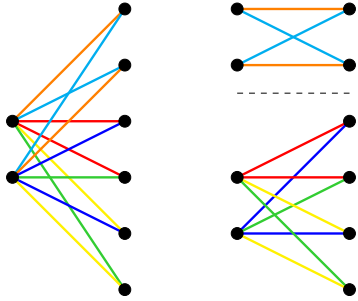

    \centering
    \includestandalone[width=0.3\textwidth]{reducible-conflict}
    \caption{An example of a reducible coloring of $K_{2,6}$}
    \label{fig:reducible}
\end{figure}

We now define the conflict system $\cC$.

\begin{definition}\hfill
\begin{itemize}
    \item A matching $C = \{T_1, \dots, T_j\}$  in $\cH_1$ with $j\ge 2$ is called a {\em conflict (with respect to $\mathcal{H}_1$)} if  $E_C$ contains a bad copy $K$ of  $K_{2,r}$ for some $r\ge 1$, so that $K$ uses at least one edge from each tile in $C$. 
 \item  A conflict $C = \{T_1, \dots, T_j\}$  in $\cH_1$ is  {\em irreducible} if $E_C$ contains (at least one) irreducible bad copy $K_C$ of  $K_{2,r}$ for some $r\ge 1$, so that $K_C$ uses at least one edge from each tile in $C$. 
\item A conflict  $C$ that is not irreducible is called {\em reducible}. 
 \item   Let $\cC$ be the set of all inclusion-minimal irreducible conflicts in $\cH_1$.
 \end{itemize}   
\end{definition}
\medskip
 
\begin{claim}\label{claim:conflictsize1}
If $C$ is a conflict with respect to $\mathcal{H}_1$ then $|C|\ge 3$.
\end{claim}
\begin{proof}
Suppose, for contradiction, there exists a conflict $C$ with respect to $\mathcal{H}_1$ with $|C|\leq 2$. By definition $|C|\geq 2$, and hence we assume $C = \{T_1,T_2\}$ where $T_1$ and $T_2$ are tiles in $\mathcal{H}_1$. Since $C$ is a conflict, $E_C = E(S_{T_1})\cup E(S_{T_2})$ must contain a bad copy $K$ of $K_{2,r}$ for some $r\geq 1$, where $S_{T_i}$ is the copy of the graph $K_{t+1,t+1}-(t+1)K_2$ corresponding to $T_i$, so that $K$ uses at least one edge from each $T_i$. For $i=1,2$, let $E_K^i=E(K)\cap E(S_{T_i})$ and  $V_K^i=V(K)\cap V(S_{T_i})$. Since $C$ is a matching in $\mathcal{H}_1$, $E_K^1\cap E_K^2=\emptyset$. Let $c_i$ be the color of the edges of $S_{T_i}$.

If $c_1 = c_2=c$, then $V_K^1\cap V_K^2=\emptyset$ since otherwise, for  $v\in V_K^1\cap V_K^2$ we have $v^c\in T_1\cap T_2$,  contradicting the fact that $C$ is a matching. But since there is no way to decompose a connected graph into two vertex-disjoint nonempty subgraphs, we must have $V_K^1\cap V_K^2\neq \emptyset$, a contradiction.

So assume $c_1\neq c_2$. Let $\{y_1,\dots,y_r\}$ be the vertices of $K$ in the vertex side of size $r$ and let $\{x_1,x_2\}$ be the vertices in vertex side of size $2$. Since $C$ is a matching in $\mathcal{H}_1$ and $E_K^i \neq \emptyset$, for some $i\in \{1,2\}$ we must have $|\{x_1,x_2\}\cap V_{K}^i|=1$. Without loss of generality, let $\{x_1,x_2\}\cap V_{K}^1 = \{x_1\}$, that is, $E_K^1$ induces a star centered at $x_1$ with at least two leaves, say $y_i$ and $y_j$ (note that $E_K^1$ must have at least two leaves, for otherwise $K$ has an odd color class;  this implies in particular $r\ge 2$). This entails $x_2y_i, x_2y_j \in E_K^2$. Thus, $y_iy_j\in T_1\cap T_2$, contradicting the fact that $C$ forms a matching.
\end{proof}

\begin{claim}\label{claim:confcontainirr}
    Every conflict $C$ contains a conflict $C'$ such that $C'\in \cC$.
\end{claim}
\begin{proof}
Let $C=\{T_1,\dots, T_j\}$ be a conflict.  
If $C$ is irreducible, then since $\mathcal{C}$ contains all inclusion-minimal irreducible conflict, there must exist a conflict $C'\in \mathcal{C}$ such that $C'\subseteq C$. 

Assume for contradiction that there exists a reducible conflict $C$ that does not contain a conflict $C'$ with $C'\in \mathcal{C}$. Choose the smallest such $C=\{T_1,\dots, T_j\}$. By Claim \ref{claim:conflictsize1}, $j\ge 3$. Since $C$ is a conflict, there exists a bad copy $K$ of $K_{2,r}$, for some $r\geq 1$, so that $K$ uses at least one edge from each tile in $C$. Since $C$ is reducible $K$ must be  reducible,  and thus  $r\ge 2$ and  $K$ can be decomposed into bad copies $K_1,\dots ,K_k$ of $K_{2,r_1}, \dots,K_{2,r_k}$, respectively, for some $k\geq 2$, each with the same vertex side of size $2$ as $K$, say $\{x_1,x_2\}$. Moreover, we can choose such a $k$ so that each $K_i$ is irreducible for all $1\le i\le k$. Let $C^i=\{T_1^i,\dots,T_{\ell_i}^i\}$ be the tiles of $C$ that have at least one edge in $K_i$. 

Note that there is at most one $i\in [k]$ with $|C^i|=1$. To see this, observe that if $C^i=\{T^i_1\}$, then $x_1x_2\in T^i_1$. Thus, if $|C^i|=|C^j|=1$ for distinct $i,j\in [k]$, then $x_1x_2\in T^i_1\cap T^j_1$ contradicting that $C$ is a matching. Now consider any $C^i$ with $|C^i|>1$. Since $K_i$ is irreducible and $|C^i|\geq 2$, $C^i$ is an irreducible conflict. Thus, there exists some $C'\subseteq C^i\subseteq C$ with $C'\in \mathcal{C}$. 
\end{proof}

\begin{claim}\label{claim: matching H1 is a good coloring}
If $M$ is a matching in $\cH_1$ and $M$ contains no conflict of $\cC$ as a subset, then the edge-coloring of $K_{n,n}$ corresponding to $M$ has no bad copies of $K_{2,t}$. 
\end{claim}
\begin{proof}
Let $M = \{T_1,\dots,T_m\}$ be a matching in $\mathcal{H}_1$ that contains no conflict of $\mathcal{C}$ as a subset, and suppose for contradiction that the edge-coloring of $K_{n,n}$ corresponding to $M$ has a bad copy $K$ of $K_{2,t}$. 
Note that the colored edges of $K$ correspond to at least $2$ distinct tiles in $M$, since the graph $K_{t+1,t+1}-(t+1)K_2$ does not contain $K_{2,t}$ as a subgraph.  Let  $C=\{T_1,\dots,T_j\}$ be the tiles of $M$ such that $E(K)\cap E(S_{T_i}) \neq \emptyset$. Then
$j\geq 2$. Thus, by definition, $C$ is a conflict. Now by Claim \ref{claim:confcontainirr}, $C$ must contain a conflict $C'$ such that $C'\in \mathcal{C}$, a contradiction to the assumption on $M$.
\end{proof}

We now wish to check that Conditions \ref{C1}-\ref{C3} hold for $\cC$.  First, note that \ref{C1} holds since, by Claim \ref{claim:conflictsize1} and the definition, $3 \leq |C| \leq 2t$. 
In the following two subsections, we prove conditions \ref{C2} and \ref{C3}. It will be easier to start with \ref{C3}.

\subsubsection{Proving condition \ref{C3}}

Given a conflict $C\in \mathcal{C}$, we arbitrarily choose and fix an irreducible bad copy $G_C$ of $K_{2,r}$ so that $G_C$ has at least one edge from every tile in $C$ and where $2\leq r\leq t$. Note $r\geq 2$ since by claim \cref{claim:conflictsize1}, $|C|\geq 3$ and hence the $K_{2,r}$ needs at least $3$ edges. Also note that there is a constant number of ways to choose $G_C$, and this constant depends only on $t$ and $j\le t$.
If $C' \subset C$, we denote by $G_{C'}$ the subgraph of $G_C$ whose edges correspond to $C'$.

Recall that $\mathcal{C}^{(j)}$ denotes the set of conflicts in $\mathcal{C}$ of size $j$.
Let $2\leq j'\leq j-1$ and let $C'=\{T_1,...,T_{j'}\}$ be a set of fixed tiles in $\cH_1$ that form a matching. We want to compute an upper bound for the number of conflicts $C\in \mathcal{C}^{(j)}$ with $C'\subset C$.  For every conflict $C$ having $C'$ as a subset, we have fixed a bad copy $G_C$ and a subgraph $G_{C'}$ of $G_C$.  
We will proceed as follows:  let $\mathcal{E}=\{E_1, \dots, E_j\}$ be a partition of $E(G_{C})$ such that each part $E_i$ corresponds to a  distinct tile of $C$, and for all $i\in [j]$, $E_i\in \mathcal{E}$ corresponds to the tile $T_i \in C'$. Note that there is a constant number of such partitions, and this constant depends only on $t$.  We will find an adequate ordering $E_{j'+1},...,E_{j}$ on $\mathcal{E}\setminus \{E_1,...,E_{j'}\}$ and add the edges in $E_{j'+1}\cup ...\cup E_{j}$ to $G_{C'}$ in that order, while at each step adding the corresponding tile $T_i$ to $C'$  and updating the new fixed tiles $C'=C'\cup \{T_i\}$ and the corresponding new $G_{C'}= G_{C'\cup \{T_i\}}$. We will continue this until we have added all the edges of $E_{j'+1}\cup ...\cup E_{j}$ to $G_{C'}$ and an irreducible $K_{2,r}$ is built. 

Since at each step there is only a constant number of ways to choose $G_{C'}$ from the corresponding edges of the currently fixed tiles $C'$, the codegree of our initial $j'$ tiles will be bounded by some constant times the number of ways to complete $E_{C'}$  to an irreducible bad copy of $K_{2,r}$, where the constant depends only on $t$ and on $j-j' \le t$.

For a set of edges $E$, let $V(E)$ be the set of all vertices of the edges in $E$.

Let $\{y_1,...,y_s\}$ be vertices of $G_{C'}$ in the vertex side that will  be completed to $r$ vertices,  where $s\leq r$. Without loss of generality, $\{y_1,...,y_s\} \subset Y$. Let $X(G_{C'}) = X\cap V(G_{C'})$. Then $X(G_{C'})$ contains either one or two vertices. 

\begin{claim}\label{lemma:first} If $|X(G_{C'})|=2$ and $j-j'\geq 2$ then there exists  $E\in \mathcal{E}\setminus \{E_1,...,E_{j'}\}$ such that $E$ can be completed to a tile in $O(n^{2t+1})$ ways. If $j-j'=1$, then $\{E_j\} = \mathcal{E}\setminus \{E_1,...,E_{j-1}\}$ can be completed to a tile in $O(n^{2t})$ ways.  
\end{claim}

\begin{proof}
Suppose $X(G_{C'})=\{x_1,x_2\}$. 
First, assume $j-j'\geq 2$.
We claim that at least one of the following two possibilities occurs:
\begin{enumerate}
    \item[(1)] either there exists an odd color class in $G_{C'}$, or
    \item[(2)] there exists  $E\in \mathcal{E}\setminus \{E_1,...,E_{j'}\}$ such that $|V(E) \cap  V(E_1\cup...\cup E_{j'})| \ge 2$.
\end{enumerate}
Indeed, suppose there is no odd color class in $G_{C'}$. Observe that in this case $s\ge 2$ because if $s=1$ then  $G_{C'}$ is a $K_{1,2}$, and since $j'\geq 2$, $|E_1|=|E_2|=1$ and both of these edges have the same color, a contradiction to the fact that ${T_1,T_2}$ is a matching.

Moreover, $G_{C'}$ must contain a leaf $v \in \{y_1, \dots, y_s\}$, since otherwise $G_{C'}$ and $G_C - G_{C'}$ are  bad copies of $K_{2,s}$ and $K_{2,r-s}$, respectively, contradicting the fact that  $G_C$ is irreducible. Without loss of generality, let $v = y_1$ with neighbor $x_1$ in $G_{C'}$. Let $E\in \mathcal{E}\setminus \{E_1,...,E_{j'}\}$ such that $x_2y_1\in E$. Thus, we achieve (2).

Now, if (1) occurs, then there must be some $E\in \mathcal{E}\setminus \{E_1,...,E_{j'}\}$ such that the color of $E$  is present in $G_{C'}$. Since $X(G_{C'})=\{x_1,x_2\}$, $E$ has its color and at least one vertex fixed from $G_{C'}$. Thus, there are $O(n^{2t+1})$ ways to complete $E$ to a tile (we have to choose $2t+1$ more vertices). If (2) occurs, then $E\in \mathcal{E}\setminus \{E_1,...,E_{j'}\}$ with at least two endpoints present in $G_{C'}$ has at least two vertices fixed from $G_{C'}$ and again there are $O(n^{2t+1})$ ways to complete $E$ to a tile (we have to choose $2t-2$ vertices and a color).

Finally, suppose $j-j' =1$. By the degree-sum, there are an even number of odd degree vertices in $G_{C'}$. Moreover, any such vertex of odd degree in $G_{C'}$ must be in $V(E_j)$ since its degree is even in $G_{C}$. If there are at least $4$ odd degree vertices in $G_{C'}$, then $|V(E_j)\cap V(E_1 \cup \cdots \cup E_{j-1})|\geq 4$ and there are $O(n^{2t})$ ways to complete $E_j$ to a tile. If there are no odd degree vertices in $G_{C'}$, then $G_{C'}$ is isomorphic to $K_{2,i}$ for some $1 \leq i < r$. But then $G_C$ can be decomposed into $K_{2,i}$ and $K_{2,r-i}$, which share the same vertex part of size $2$, contradicting that $G_C$ is irreducible. Thus, we may assume there are exactly two odd degree vertices $u$ and $v$ in $G_{C'}$.

Up to symmetry and without loss of generality, there are $3$ possibilities for $u$ and $v$. Either (i) $u=x_1$ and $v=x_2$, (ii) $u= x_1$ and $v=y_1$, or (iii) $u=y_1$ and $v=y_2$. In case (i), since $x_1$ and $x_2$ are the only odd degree vertices in $G_{C'}$, $G_{C'}$ is again isomorphic to $K_{2,i}$ for some $1\leq i < r$, a contradiction. 

In cases (ii), it suffices to show that $|E_j|$ is odd since together with $|V(E_j)\cap V(E_1 \cup \cdots \cup E_{j-1})|=2$, this implies there are $O(n^{2t})$ ways to complete $E_j$ to a tile. To see this, observe that  $|E_j|= 2r-|E(G_{C'})| = 2r - \deg_{G_{C'}}(x_1)- \deg_{G_{C'}}(x_2)$ Since $\deg_{G_{C'}}(x_1)$ is odd degree and $\deg_{G_{C'}}(x_2)$ is even, we see that $|E_j|$ is odd.

Finally, assume we are in case (iii) where $y_1$ and $y_2$ are odd degree vertices in $G_{C'}$. This implies $y_1$ and $y_2$ are leaves in $G_{C'}$. Since \[\deg_{G_{C'}}(x_1)+\deg_{G_{C'}}(x_2) = \sum_{y_i \in V(G_{C'})\setminus \{x_1,x_2\}}\deg_{G_{C'}}(y_i) < 2r,\] $|\{x_1,x_2\}\cap V(E_j)|\geq 1$. Thus, $|V(E_j)\cap V(E_1\cup \cdots \cup E_{j-1})|\geq 3$ and there are $O(n^{2t})$ ways to complete $E_j$ to a tile. This completes the claim.
\end{proof}

When $|X(G_{C'})|=2$,  by iteratively applying \cref{lemma:first} after each unfixed tile is chosen, we get \[\Delta_{j'}(\mathcal{C}^{(j)}) = O\left(n^{(j-j'-1)(2t+1)+2t}\right) = O\left(n^{(2t+1)(j-j'-\frac{1}{2t+1})}\right).\]

which would satisfy \ref{C3} for any $\varepsilon\in \left(0,\frac{1}{2t+2} \right)$. We now have to consider the case when $|X(G_{C'})|=1$, say $X(G_{C'})=\{x_1\}$. Then $G_{C'}$ is a star centered at $x_1$.

It will be helpful to refer to tiles in $C$ whose corresponding colors are the same. To this end, if the color corresponding to $T_i$ is the same as the color corresponding to $T_j$, we call $T_i$ and $T_j$ \textit{partners}. Note that if $T_i\in C$ contributes an odd number of edges to the bad $K_{2,r}$, then by definition of being bad, $T_i$ must have a partner in $C$. Also, notice that no tile can have more than one partner since every tile with corresponding color $c$, must contain $x_1^c$ or $x_2^c$. Thus, if there is a tile with at least two partners, by the pigeonhole principle, there exist two tiles containing either $x_1^c$ or $x_2^c$, contradicting that the tiles form a matching in $\mathcal{H}_1$.

\begin{claim}\label{lemma: fixed star}
Let $G_{C'} \cong K_{1,s}$ be a star centered at $x_1$. Then there exists an ordering $E_{j'+1},\dots,E_{j}$ of $\mathcal{E}\setminus \{E_{1},\dots, E_{j'}\}$ such that one the following must occur:
\begin{enumerate}
    \item[(1)] There are $O(n^{2t+1})$ ways to complete $E_{j'+i}$ to a tile for all $1\leq i \leq j-j'-1$ and $O(n^{2t})$ ways to complete $E_{j}$ to a tile.
    \item[(2)] There are $O(n^{2t+2})$ ways to complete $E_{j'+1}$ to a tile, $O(n^{2t+1})$ ways to complete $E_{j'+i}$ to a tile for all $2\leq i \leq j-j'$, but moreover there is $O(n^{2t-1})$ ways to complete $E_q$ to a tile for some $q\in [j]\setminus [j']$.
    \item[(3)] There are $O(n^{2t+2})$ ways to complete $E_{j'+1}$ to a tile, $O(n^{2t+1})$ ways to complete $E_{j'+i}$ to a tile for all $2\leq i \leq j-j'$, but moreover there are $O(n^{2t})$ ways to complete $E_q$ to a tile for some $q\in [j-1]\setminus [j']$, and $O(n^{2t})$ ways to complete $E_{j}$ to a tile.
\end{enumerate} 
\end{claim}

\begin{proof}
We prove this in three steps, but first, we establish some notation. For any sequence $E_{j'+1},...,E_{j}$ of $\mathcal{E}\setminus \{E_1,...,E_{j'}\}$ and for $0 \leq i \leq j - j'$, we let $\ell_i$ denote the number of leaves in $G_{C'}$ after the $E_{j'+i}$ has been added to $G_{C'}$. Note that $\ell_0 = s \geq 2$, and $\ell_{j-j'} = 0$ since $G_C \cong K_{2,t}$. For $1 \leq i \leq j - j'$, we say that $E_{j'+i}$ \emph{eliminates} $z$ leaves if $\ell_i = \ell_{i-1} - z$, and we say that it \emph{creates} $z$ leaves if $\ell_i = \ell_{i-1} + z$.

\medskip

\noindent \textbf{Step 1: Choosing $E_{j'+1}$.} Suppose either of the following circumstances occurs:
\begin{enumerate}
    \item[(i)] There exists an $E\in \mathcal{E}\setminus \{E_1,\dots,E_{j'}\}$ such that at least two endpoint of $E$ are present in $G_{C'}$.
    \item[(ii)] There exists an odd color class in $G_{C'}$ coming from $E_{i}$ for some $i\in [j']$, and the partner of $E_{i}$, say $E\in \mathcal{E}\setminus \{E_1,\dots,E_{j'}\}$ has at least one endpoint present in $G_{C'}$.
\end{enumerate}
In either case, there are $O(n^{2t+1})$ ways to complete $E$ to a tile. After $E$ is completed to a tile, if $G_{C'}$ is no longer a star centered at $x_1$. By applying \cref{lemma:first} to $\mathcal{E} \setminus \{E_{1},...,E_{j'},E\}$, we can obtain the sequence desired in (1) with $E = E_{j'+1}$. If $G_{C'}$ is still a star centered at $x_1$, we repeat the argument again.

We now assume no such $E\in \mathcal{E}\setminus \{E_1,...,E_{j'}\}$ exists. In this case, we define $E_{j'+1}$ to be the edge set containing $x_2y_1$. Note that there are at most $O(n^{2t+2})$ ways to complete $E_{j'+1}$ to a tile, and after $E_{j'+1}$ is completed to a tile, $G_{C'}$ has $|X(G_{C'})|=2$. By iteratively applying \cref{lemma:first}, there exists a sequence $E_{j'+2}, \dots, E_{j}$ such that there are $O(n^{2t+1})$ ways to complete $E_i$ to a tile for all $j'+2 \leq i \leq j $, and there are $O(n^{2t})$ ways to complete $E_{j}$ to a tile.  

What remains to be done in order to prove the existence of (2) or (3) is to guarantee the existence of the desired edge set $E_q$. We accomplish this by analyzing how leaves are created and eliminated. Note that after $E_{j'+1}$ is added to $G_{C'}$, $|X(G_{C'}|=2$.

\medskip

\noindent\textbf{Step 2: Analyzing leaves.} Since we assumed $E_{j'+1}$ did not satisfy (i) and (ii), we have $\ell_1 \geq \ell_0 - 1$ and if $|E_{j'+1}|= 1$ then the color of $E_{j'+1}$ was not already present in $G_{C'}$. In this case, we can assign the partner of $E_{j'+1}$ as $E_{j'+2}$. Note that there are either $O(n^{2t})$ ways to complete $E_{j'+2}$ to a tile and we have (3) with $q=j'+2$ or there are $O(n^{2t+1})$ ways to complete $E_{j'+2}$ to a tile and $E_{j'+2}$ must create at least one leaf. Thus, we may assume that either $\ell_1 \geq \ell_0$ or $\ell_2 \geq \ell_0$. Since $\ell_0 = s \geq 2$ and $\ell_{j-j'} = 0$, at some point, leaves must be eliminated. This must happen in one of the following ways.  

Suppose for some $i\in [j-j']$, there exists some $E_{j'+i}$ that eliminates $z\geq 2$ leaves. Then there are $O(n^{2t+2-z})$ ways to complete $E_{j'+i}$ to a tile (see left of \cref{fig:leaf-combos}). We call $E_{j'+i}$ a \textit{$z$-leaf set}. Next, suppose $E_{j'+i}$ eliminates one leaf. If $|E_{j'+i}|\geq 2$ then there are $O(n^{2t-1})$ ways to complete $E_{j'+i}$ into a tile and we call $E_{j'+i}$ a $(1,2^+)$ set. Now suppose $|E_{j'+i}|=1$. If the partner of $E_{j'+i}$ is already present in $G_{C'}$ prior to $E_{j'+i}$, then there are $O(n^{2t})$ ways to complete $E_{j'+i}$ to a tile. We call $E_{j'+i}$ a \textit{(1,1)-leaf set}. If the partner of $E_{j'+i}$ is not already present in $G_{C'}$ then there are $O(n^{2t+1})$ ways to complete $E_{j'+i}$ to a tile and we let the partner of $E_{j'+i}$ be $E_{j'+i+1}$. If $E_{j'+i+1}$ is a $z$-leaf-set, or if $E_{j'+i+1}$ eliminates $z=0$ leaves, then there are $O(n^{2t+1-z})$ ways to complete $E_{j'+i+1}$ to a tile. We call $(E_{j'+i},E_{j'+i+1})$ a \textit{$(z+1)$-leaf pair} (see right of \cref{fig:leaf-combos}). 

The only other case is if $
E_{j'+i+1}$ creates at least one leaf and thus, $\ell_{i+1}\geq \ell_{i-1}$. Table \ref{tab:leaf} summarizes the four possibilities for how leaves are eliminated, and the corresponding number of ways to complete $E_{j'+i}$ (and $E_{j'+i+1}$) to a tile.

\begin{table}[h!]
\renewcommand{\arraystretch}{1.5}
    \centering
    \begin{tabular}{|c|c|c|}
       \hline
       $E_{j'+i}$ or $(E_{j'+i},E_{j'+i+1})$ & number of leaves eliminated & ways to complete to a tile  \\
       \hline \hline
       $z$-leaf set & $z\geq 2$ & $O(n^{2t+2-z})$ \\
       \hline
       $(1,1)$-leaf set & $1$ & $O(n^{2t})$ \\
       \hline
       $(1,2^+)$-leaf set & $1$ & $O(n^{2t-1})$ \\
       \hline
       $(z+1)$-leaf pair & $z+1\geq 1$ & $O(n^{2t+1})$, $O(n^{2t+1-z})$\\
       \hline
       
    \end{tabular}
    \caption{A summary of the ways leaves can get eliminated as $G_{C'}$ becomes $G_{C}$.}
    \label{tab:leaf}
\end{table}

\noindent\textbf{Step 3: Finding $E_q$.} If $s \geq 3$, then one of the following must occur. First, either we have a $z$-leaf set or a $z$-leaf pair for some $z\geq 3$. Both of which have some $E_{j'+i}$ with $O(n^{2t-1})$ ways to complete it to a tile. Letting $q=j'+i$, we achieve (2). Otherwise, we must have at least two sets or pairs $E_{j'+i}$ or $(E_{j'+i},E_{j'+i+1})$ and $E_{j'+i'}$ or $(E_{j'+i'},E_{j'+i'+1})$ that each eliminate at least one leaf. Since each contains some $E_q$ with $O(n^{2t})$ ways to complete it to a tile (one of which can be assumed to be $E_{j}$ since the final $E_j$ must eliminate a leaf by irreducibility), we achieve (3).  

If $s = 2$, then $C'$ consists of two tiles $T_1$ and $T_2$ such that $E_1$ and $E_2$ are single edges, $x_1y_1$ and $x_1y_2$ respectively, without loss of generality. Moreover, each edge must be colored with distinct colors. Let $E_{j'+1}$ be the partner of $E_1$. $E_{j'+1}$ must contain some edge $x_2y_i$ for some $i\in [r]\setminus \{1\}$. If $y_i = y_2$, then there are $O(n^{2t+1})$ ways to complete $E_{j'+1}$ to a tile, and then by applying \cref{lemma:first}, we can achieve (1). If $y_i \neq y_2$, then there are $O(n^{2t+2})$ ways to complete $E_{j'+1}$ to a tile, but now $\ell_1 \geq 3$. By applying the same logic as above, either (2) or (3) must be achieved, thus completing the proof.
\end{proof}

\begin{figure}[h!]
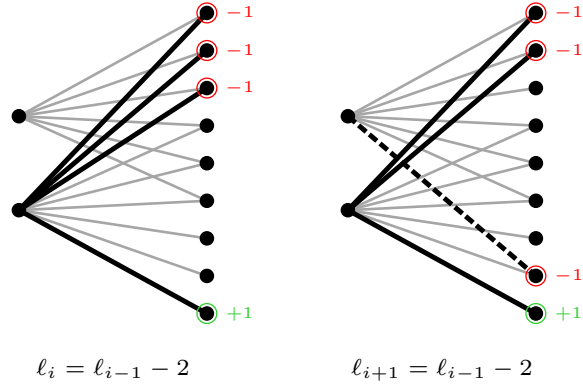

    \centering
    \includestandalone[width=0.5\textwidth]{leaf-combos}
    \caption{An example of adding $E_{j'+i}$ (black) to $G_{C'}$ (gray). On the left: $E_{j'+i}$ is a $2$-leaf set. On the right: $(E_{j'+i},E_{j'+i+1})$ is a $2$-leaf pair ($E_{j'+i}$ dashed and $E_{j'+i+1}$ solid). The $-1$ indicates the elimination of a leaf and the $+1$ indicates the creation of a leaf.}
    \label{fig:leaf-combos}
\end{figure}

Now by \cref{lemma: fixed star}, we get in all three cases,
\begin{enumerate}
    \item[(1)] $\Delta_{j'}(\mathcal{C}^{(j)})\leq O(n^{(2t+1)(j-j'-1)+2t}) = O\left(n^{(2t+1)(j-j'-\frac{1}{2t+1})}\right)$.
    \item[(2)] $\Delta_{j'}(\mathcal{C}^{(j)})\leq O(n^{2t+2+(2t+1)(j-j'-2)+2t-1})=O\left(n^{(2t+1)(j-j'-\frac{1}{2t+1})}\right)$.
    \item[(3)] $\Delta_{j'}(\mathcal{C}^{(j)})\leq O(n^{2t+2+(2t+1)(j-j'-3)+4t})=O\left(n^{(2t+1)(j-j'-\frac{1}{2t+1})}\right)$.
\end{enumerate}

We now satisfy that $\Delta_{j'}(\cC^{(j)}) \leq d^{j-j'-\varepsilon}$ for any $\varepsilon\in (0,\frac{1}{2t+2})$. Thus, \ref{C3} is satisfied.

\subsubsection{Proving 
condition \ref{C2}}

Next, we prove condition \ref{C2} in the following lemma.

\begin{lemma}
$\Delta(\mathcal{C}^{(j)})\leq \ell d^{j-1}$ for all $3\leq j\leq \ell$.
\end{lemma}

\begin{proof}
Let $C' = \{T_1\}$ be a tile from $\mathcal{H}_1$. We want to provide an upper bound for the number of conflicts $C$ that contain $C'$. Let $E_1\in \mathcal{E}$ be the edge set corresponding to $T_1$. First suppose $|X(G_{C'})|=1$, say $X(G_{C'})=\{x_1\}$. Let $y_1\neq x_1$ be a leaf of $G_{C'}$. Let $E_2\in \mathcal{E}\setminus \{E_1\}$ be the edge set  which contains $x_2y_1$. Since $y_1$ is present in $G_{C'}$, there are at most $O(n^{2t+2})$ ways to complete the tile.

Next, suppose $|X(G_{C'})|=2$. Then letting $E_2$ be any edge set from $\mathcal{E}\setminus \{E_1\}$, we see again there are at most $O(n^{2t+2})$ ways to complete $E_2$ to a tile. 

Once $E_1$ and $E_2$ are fixed and completed to tiles $T_1$ and say $T_2$, we can apply the same upper bound as the one found in \ref{C3} with $C' = \{T_1,T_2\}$ to obtain the following desired bound. \[\Delta(\cC^{(j)})\leq O\left(n^{2t+2+(2t+1)(j-2-\frac{1}{2t+1})}\right) = O(n^{2t+2+(2t+1)j - 4t-2-1}) = O(n^{(2t+1)(j-1)}) = O(d^{j-1}). \] Thus, $\Delta(\mathcal{C}^{(j)})\leq \ell d^{j-1}$ for some sufficiently large constant $\ell$.
\end{proof}

\subsection{The conflict system $\cD$}\label{sec:conD}

We now define conflict system $\cD$, similarly to $\cC$.

\begin{definition}\hfill
\begin{itemize}
    \item A matching $D = \{T_1, \dots, T_j\}$  in $\cH$ with $j\ge 2$ is called a {\em conflict (with respect to $\mathcal{H}$)} if  $E_D$ contains a bad copy $K$ of $K_{2,r}$ for some $r\ge 1$, so that $K$ uses at least one edge from each tile in $D$. 

    \item  A conflict $D = \{T_1, \dots, T_j\}$  in $\cH$ is  {\em irreducible} if $E_D$ contains (at least one) irreducible bad copy $K_D$ of  $K_{2,r}$ for some $r\ge 1$, so that $K_D$ uses at least one edge from each tile in $D$.

    \item A conflict $D$ that is not irreducible is called {\em reducible}.

    \item   Let $\cD$ be the set of all inclusion-minimal irreducible conflicts in $\cH$.

\end{itemize}

\end{definition}

Given a conflict $D\in \mathcal{D}$, we let $G_D$ denote the $K_{2,r}$ subgraph of $K_{n,n}$ which is colored with an irreducible bad coloring. This way $V(D)$ (tiles in $\mathcal{H}$) and $V(G_D)$ (vertices in $K_{n,n}$) can be differentiated. If we have several tiles \emph{fixed} in a conflict $D\in \mathcal{D}$, we use $D'$ to denote the fixed subconflict of $D$ and $G_{D'}$ to denote the fixed subgraph of $G_D$.

Observe that by this definition, every conflict in $\cD$ has at least two edges which are tiles from $\cH_2$, but no more than $2t$, and hence property \ref{D1} is satisfied.

\begin{claim}\label{claim:conflictsize2}
    If $F$ is a conflict with respect to $\cH_1$ or $\cH$, then $|F| \geq 3$.
\end{claim}

\begin{proof}
Let $F$ be a conflict with respect to $\mathcal{H}$ or $\mathcal{H}_1$ and let $K$ be an irreducible bad $K_{2,r}$ corresponding to $F$.

If $F$ is a conflict with respect to $\mathcal{H}_1$ then by \cref{claim:conflictsize1}, we have $|F|\geq 3$. Thus, we may assume $F$ is a conflict with respect to $\mathcal{H}$ and hence must contain at least two tiles from $\mathcal{H}_2$. We now assume for contradiction that $F = \{T_1,T_2\}$ where both $T_1,T_2\in \mathcal{H}_2$ and hence $K$ is isomorphic to $K_{2,1}$.

Without loss of generality, $T_1 = \{ x, y_1, c_1 \}, T_2 = \{ x, y_2, c_2 \}$.  If $c_1 \neq c_2$, then $K$ is not bad, a contradiction. If $c_1 = c_2 = c$, then $x^c \in T_1 \cap T_2$, contradicting that $T_1$ and $T_2$ form a matching. 
\end{proof}

We now prove the following claims that are analogous to \cref{claim:confcontainirr} and \cref{claim: matching H1 is a good coloring}. We provide the statements and proofs below for completeness. 

\begin{claim}\label{claim: F contains conflict-system}
    Every conflict $F$ (with respect to either $\mathcal{H}$ or $\mathcal{H}_1$) contains a conflict $F'$ such that $F'\in \mathcal{C}\cup \mathcal{D}$.
\end{claim}

\begin{proof}
Let $F=\{T_1,\dots, T_j\}$ be a conflict.  
If $F$ is irreducible, then since $\mathcal{C}\cup \mathcal{D}$ contains all inclusion-minimal irreducible conflicts, there must exist a conflict $F'\in \mathcal{C}\cup \mathcal{D}$ such that $F'\subseteq F$. 

Assume for contradiction that there exists a reducible conflict $F$ that does not contain a conflict $F'$ with $F'\in \mathcal{C}\cup \mathcal{D}$. Choose the smallest such $F=\{T_1,\dots, T_j\}$. By Claim \ref{claim:conflictsize2}, $j\ge 3$. Since $F$ is a conflict, there exists a bad copy $K$ of $K_{2,r}$, for some $r\geq 1$, so that $K$ uses at least one edge from each tile in $F$. Since $F$ is reducible $K$ must be  reducible,  and thus  $r\ge 2$ and $K$ can be decomposed into bad copies $K_1,\dots,K_k$ of $K_{2,r_1},\dots,K_{2,r_k}$, respectively, for some $k\geq 2$, each with the same vertex side of size $2$ as $K$, say $\{x_1,x_2\}$. Moreover, we can choose such a $k$ so that each $K_i$ is irreducible for all $1\le i\le k$. Let $F^i=\{T_1^i,\dots,T_{\ell_i}^i\}$ be the tiles of $F$ that have at least one edge in $K_i$. 

Note that there is at most one $i\in [k]$ with $|F^i|=1$. To see this, observe that if $F^i=\{T^i_1\}$, then $T_1^i$ comes from $\mathcal{H}_1$ and hence $x_1x_2\in T^i_1$. Thus, if $|F^i|=|F^j|=1$ for distinct $i,j\in [k]$, then $x_1x_2\in T^i_1\cap T^j_1$ contradicting that $F$ is a matching. Now consider any $F^i$ with $|F^i|>1$. Since $K_i$ is irreducible and $|F^i|\geq 2$, $F^i$ is an irreducible conflict. Thus, there exists some $F'\subseteq F^i\subseteq F$ with $F'\in \mathcal{C}\cup \mathcal{D}$. 
\end{proof}

\begin{claim}\label{claim: matching is a good coloring}
If $M$ is a matching in $\cH$ and $M$ contains no conflict of $\cC \cup \mathcal{D}$ as a subset, then the edge-coloring of $K_{n,n}$ corresponding to $M$ has no bad copies of $K_{2,t}$. 
\end{claim}

\begin{proof}
Let $M = \{T_1,\dots,T_m\}$ be a matching in $\mathcal{H}$ that contains no conflict of $\mathcal{C}\cup \mathcal{D}$ as a subset, and suppose for contradiction that the edge-coloring of $K_{n,n}$ corresponding to $M$ has a bad copy $K$ of $K_{2,t}$.
Note that the colored edges of $K$ correspond to at least $2$ distinct tiles in $M$, since the graph $K_{t+1,t+1}-(t+1)K_2$ does not contain $K_{2,t}$ as a subgraph.  Let $F=\{T_1,\dots,T_j\}$ be the tiles of $M$ such that $E(K)\cap E(S_{T_i}) \neq \emptyset$. Then
$j\geq 2$. Thus, by definition, $F$ is a conflict. Now by Claim \ref{claim: F contains conflict-system}, $F$ must contain a conflict $F'$ such that $F'\in \mathcal{C}\cup \mathcal{D}$, a contradiction to the assumption on $M$.
\end{proof}

\subsubsection{Proving condition \ref{D2}}

We first check that $\left| \cD_x^{(j_1,j_2)} \right| \leq d^{j_1 + \varepsilon^4} \delta_P(\cH_2)^{j_2}$ for all $j_1 \in [0,O(1)], j_2 \in [2,O(1)]$.  To do so, first observe $\delta_P(\cH_2) = n^\delta$ since given any vertex in $P$ which is in a tile of $\cH_2$, there are $n^\delta$ ways to choose the color for the vertex in $P$, and tiles in $\cH_2$ correspond to colored edges of $K_{n,n}$.  To determine $\left| \cD_x^{(j_1,j_2)} \right|$, we count the number of ways to complete an edge $x \in P$ from a tile of $\cH_2$ to a conflict $D\in \mathcal{D}$ with $j_1$ tiles coming from $\cH_1$ and $j_2$ tiles coming from $\cH_2$.  Now, let $j_1^{(i)}$ denote the number of tiles from $\cH_1$ which contribute exactly $i$ edges to $G_D$ (a copy of $K_{2,r}$), where $1 \leq i \leq 2r$, $2 \leq r \leq t$.  Observe there are $\displaystyle c_r := \binom{j_1}{j_1^{(1)}, j_1^{(2)}, \dots, j_1^{(2r)}}$ many ways to choose the number of combinations of such types of tiles (that is, a number that depends on $r$, not $n$).  To complete $x$ to a conflict of appropriate size, there are $t-1$ ways to choose $r$, $O(n^r)$ ways to choose the remaining vertices in $G_D$ (fixing all graph edges in $G_D$), and $c_r$ ways to choose the number of edges each tile from $\cH_1$ can contribute to $G_D$.  For each $i \geq 2$, we can choose the $j_1^{(i)}$ tiles contributing $i$ edges in $O\left( \left(\frac{d}{n^{i-1}} \right)^{j_1^{(i)}} \right)$ ways. Now, there are $O\left( \left(\frac{d^2}{n} \right)^\frac{j_1^{(1)}}{2} \right)$ ways to choose the $j_1^{(1)}$ tiles contributing one edge (since for each color appearing in $G_D$, it must appear on an even number of edges; hence, we know one edge, and for half of the tiles, we know a color). Finally, there are $O(n^{\frac{1}{2} \delta j_2})$ ways to choose the tiles coming from $\cH_2$ since for every graph edge in a tile of $\cH_2$, there are $n^\delta$ ways to choose the color, but since the conflict corresponds to an even coloring, we need only choose colors for half of the tiles.  With the observation that $\displaystyle \sum_{i=1}^{2r} j_1^{(i)} = j_1$, we have that 
\begin{align*}
    \left| \cD_x^{(j_1,j_2)} \right| &\leq O \left(n^{r + \frac{1}{2} \delta j_2 - \frac{1}{2} j_1^{(1)} - \sum_{i=2}^{2r} (i-1) j_1^{(i)}} \right) \cdot (t-1) \cdot c_r \cdot O \left(d^{\sum_{i=1}^{2r}j_1^{(i)}} \right)\\
    &= O \left(n^{r + \frac{1}{2} \delta j_2 - \frac{1}{2} j_1^{(1)} - \sum_{i=2}^{2r} (i-1) j_1^{(i)}} \cdot d^{j_1} \right).    
\end{align*}

Now, observe that 
$$
    O \left(n^{r + \frac{1}{2} \delta j_2 - \frac{1}{2} j_1^{(1)} - \sum_{i=2}^{2r} (i-1) j_1^{(i)}} \cdot d^{j_1} \right) < O(d^{j_1 + \varepsilon^4} n^{\delta j_2} )
$$
if and only if
$$
    2r - \delta j_2 - j_1^{(1)} - \sum_{i=2}^{2r} 2 \cdot (i-1) j_1^{(i)} < 2 (2t+1) \varepsilon^4.
$$

Since the total number of edges in $G_D$ is $2r$, and each of the $j_1^{(i)}$ tiles from $\cH_1$ contribute $i$ of the $2r$ edges to $G_D$ while each of the $j_2$ tiles from $\cH_2$ contributes exactly one of the $2r$ edges to $G_D$, this inequality holds for our particular choice of $\varepsilon$ and $\delta \in \left( 1-(2t+1) \varepsilon^4, 1 \right) $. That is, $\left| \cD_x^{(j_1,j_2)} \right| \leq d^{j_1 + \varepsilon^4} n^{ \delta j_2}$, and condition \ref{D2} is satisfied.

\subsubsection{Proving condition \ref{D3}}

To obtain bounds on $\Delta_{j',0} \left( \cD_x^{(j_1,j_2)} \right)$, we count the number of ways to complete $j'$ tiles from $\cH_1$ to a conflict in $\cD_x^{(j_1,j_2)}$.  To this end, fix $j'$ tiles from $\cH_1$.  Observe also that $x$ is fixed.  There are $t-1$ ways to choose $r$ for $G_D$ corresponding to our conflict.  Let $r'$  be the number of graph vertices fixed by the $j'$ tiles and $x$ in the corresponding $G_D$ in the vertex side of size $r$.  Note that $x$ being fixed ensures at least one vertex in the vertex side of size 2 is fixed.  Lastly, let $j_1^{(i)}$ denote the number of \emph{unfixed} tiles from $\cH_1$ which contribute exactly $i$ edges to the corresponding $G_D$.  

We handle the two possible cases, $G_{D'}$ spans both $x_1$ and $x_2$ or $G_{D'}$ spans only $x_1$, with two claims.

\begin{claim}
    Suppose $G_{D'}$ spans both $x_1$ and $x_2$; that is, two vertices in the vertex side of size 2 are fixed by the $j'$ tiles or $x$. Then \ref{D3} holds.
\end{claim}

\begin{proof}
    Suppose that two vertices in the vertex side of size 2 are fixed.  
    
    Since $G_{D'}$ is irreducible, either there exists an odd color class in $G_{D'}$ or there exists an unfixed tile which contributes at least two vertices to $G_{D'}$.

    In either case, there are at most $O \left(n^{r - r'} \right)$ ways to choose the remaining vertices in $G_
    D$.  Now we know all graph vertices and edges of $G_D$, and there are at most $c_r$ ways to choose the partition of the $j_1 - j'$ tiles from $\cH_1$.  Also, observe there are $O\left( \frac{d^2}{n} \right)^{\frac{1}{2} j_1^{(1)}}$ ways to choose the tiles from $\cH_1$ that contribute only one edge, and $O\left( \frac{d}{n^{i-1}} \right)^{j_1^{(i)}}$ ways to choose the tiles from $\cH_1$ that contribute exactly $i$ edges.  Lastly, there are $O(n^{\frac{1}{2} \delta j_2})$ ways to choose the tiles from $\cH_2$.  

    Since in either case, there is either an odd color class in $G_{D'}$ or an unfixed tile intersecting $G_{D'}$, then we have over-counted slightly the number to complete at least one tile above, and we thus have
    \begin{align*}
    \Delta_{j',0} \left( \cD_x^{(j_1,j_2)} \right) 
    &\leq (t-1) \cdot c_r \cdot  O \left( n^{r - r' - 1 - \frac{1}{2} j_1^{(1)} - \sum_{i=2}^{2r} (i-1) j_1^{(i)} + \frac{1}{2} \delta j_2} \right) \cdot d^{\sum_{i=1}^{2r} j_1^{(i)}} \\
    &= O \left( n^{r - r' - 1 - \frac{1}{2} j_1^{(1)} - \sum_{i=2}^{2r} (i-1) j_1^{(i)} + \frac{1}{2} \delta j_2} \cdot d^{j_1 - j'} \right).
\end{align*}

Now, observe that 
$$
    O \left( n^{r - r' - 1 - \frac{1}{2} j_1^{(1)} - \sum_{i=2}^{2r} (i-1) j_1^{(i)} + \frac{1}{2} \delta j_2} \cdot d^{j_1 - j'} \right) 
    \leq d^{j_1 - j' - \varepsilon} n^{\delta j_2}
$$
if and only if 
$$
    r - r' - 1 - \frac{1}{2} j_1^{(1)} - \sum_{i=2}^{2r} (i-1) j_1^{(i)} - \frac{1}{2} \delta j_2 < -(2t+1) \varepsilon,
$$
which is true if and only if
$$
    2r - 2r' - 2 - j_1^{(1)} - \sum_{i=2}^{2r} 2 (i-1) j_1^{(i)} - \delta j_2 < - 2(2t+1) \varepsilon. 
$$

We claim this inequality holds.  To see this, observe $2r - 2r'$ counts the number of edges in $E(G_D)\setminus E(G_{D'})$, which has no edges fixed by the $j'$ fixed tiles or $x$. Since $j_1^{(1)} + \sum_{i=2}^{2r} 2 (i-1) j_1^{(i)}$ is at least the number of edges contributed by the unfixed tiles from $\cH_1$ and $j_2$ is exactly the number of edges contributed by unfixed tiles from $\cH_2$, including the tile containing $x$, this bound holds for our choice of $\delta$.
\end{proof}

\begin{claim}
    Suppose $G_{D'}$ spans only $x_1$; that is, only one vertex in the vertex side of size 2 is fixed by the $j'$ tiles or $x$. Then \ref{D3} holds.
\end{claim}

\begin{proof}
    Suppose now that only one vertex in the vertex side of size 2 is fixed by $x$ and the $j'$ fixed tiles; that is, $G_{D'}$ is a star with $r'$ leaves. Then there are at most $O(n^{r-r' +1})$ ways to choose the remaining vertices in $G_D$.  Now, we know all graph edges of $G_D$.  There are at most $c_r$ ways to choose to partition the $j_1 - j'$ tiles from $\cH_1$.  Also, observe there are $\left( \frac{d^2}{n} \right)^{\frac{1}{2} j_1^{(1)}}$ ways to choose the tiles from $\cH_1$ that contribute only one edge, and $\left( \frac{d}{n^{i-1}} \right)^{j_1^{(i)}}$ ways to choose the tiles from $\cH_1$ that contribute exactly $i \geq 2$ edges.  Lastly, there are $n^{\frac{1}{2} \delta j_2}$ ways to choose the tiles from $\cH_2$.  Hence, we have 
\begin{align*}
    \Delta_{j',0} \left( \cD_x^{(j_1,j_2)} \right) 
    &\leq (t-1) \cdot c_r \cdot  O \left( n^{r - r' + 1 - \frac{1}{2} j_1^{(1)} - \sum_{i=2}^{2r} (i-1) j_1^{(i)} + \frac{1}{2} \delta j_2} \right) \cdot d^{\sum_{i=1}^{2r} j_1^{(i)}} \\
    &= O \left( n^{r - r' + 1 - \frac{1}{2} j_1^{(1)} - \sum_{i=2}^{2r} (i-1) j_1^{(i)} + \frac{1}{2} \delta j_2} \cdot d^{j_1 - j'} \right).
\end{align*}

Now, observe that 
$$
    O \left( n^{r - r' + 1 - \frac{1}{2} j_1^{(1)} - \sum_{i=2}^{2r} (i-1) j_1^{(i)} + \frac{1}{2} \delta j_2} \cdot d^{j_1 - j'} \right) \leq d^{j_1 - j' - \varepsilon} n^{\delta j_2}
$$
if and only if 
$$
    r - r' + 1 - \frac{1}{2} j_1^{(1)} - \sum_{i=2}^{2r} (i-1) j_1^{(i)} - \frac{1}{2} \delta j_2 < -(2t+1) \varepsilon,
$$
which is true if and only if 
$$
    2r - 2r' + 2 - j_1^{(1)} - \sum_{i=2}^{2r} 2 (i-1) j_1^{(i)} - \delta j_2 < - 2(2t+1) \varepsilon. 
$$

As before, $2r - 2r'$ counts the number of edges in $E(G_D)\setminus E(G_{D'})$, consisting of edges that are not fixed by $x$ nor the $j'$ tiles from $\cH_1$.  Now, $j_2$ counts at least one edge in $G_{D'}$ that is fixed by $x$.  Additionally, since $j' \geq 1$, then $r' \geq 2$.  If $r' > 2$, then there are at least three edges being contributed by the $j_1 - j'$ and $j_2$ tiles to $G_D$ (i.e., not being counted as the $2r - 2r'$ edges), and the inequality holds for our choice of $\delta$.

If $r' = 2$, then $j' = 1$, and further, this one fixed tile contributes only one edge to $G_D$.  Hence, since conflicts correspond to even colorings, there must be another tile of the same color, and we actually only need 
$$
    2r - 2r' - j_1^{(1)} - \sum_{i=2}^{2r} 2 (i-1) j_1^{(i)} - \delta j_2 < - 2(2t+1) \varepsilon. 
$$

Thus, there are at least two edges being contributed by the unfixed tiles, plus the edge $x$, and again this inequality holds, and we have proven \ref{D3} holds in all cases.  
\end{proof}
 
\subsubsection{Proving condition \ref{D4}}

Lastly, we wish to show $|\cD_{x,y}^{(j_1,j_2)}| \leq d^{j_1 - \varepsilon} \delta_P(\cH_2)^{j_2}$.  Similar to above, we do so by counting the number of ways to complete two fixed graph edges $x,y$ to a conflict $D\in  \cD$ with $j_1$ tiles coming from $\cH_1$ and $j_2$ tiles coming from $\cH_2$.  As before, there are $t-1$ ways to choose $r$ for $G_D$.  Letting $j_1^{(i)}$ be the number of tiles from $\cH_1$ which contribute exactly $i$ edges to $G_D$, there are $c_r$ ways to choose the partition of $j_1$, and once all vertices (and hence edges) of $G_D$, $O \left( \frac{d^2}{n} \right)^\frac{j_1^{(1)}}{2}$ ways to choose the tiles contributing exactly one edge to $G_D$ and $O \left( \frac{d}{n^{i-1}} \right)^{j_1^{(i)}}$ ways to choose the tiles contributing exactly $i$ edges to $G_D$ for $2 \leq i \leq 2r$.  

Note that $x,y$ are distinct edges in $K_{n,n}$, so at least three vertices of $G_D$ are fixed, and we must choose at most $r - 1$ more vertices.  Together with the above, this implies
$$
    |\cD_{x,y}^{(j_1,j_2)}|
    \leq O \left( n^{r - 1 - \frac{1}{2} j_1^{(1)} - \sum_{i=2}^{2r} (i-1) j_1^{(i)} + \frac{1}{2} \delta j_2} \cdot d^{j_1} \right).
$$

Hence, the inequality $|\cD_{x,y}^{(j_1,j_2)}| \leq d^{j_1 - \varepsilon} \delta_P(\cH_2)^{j_2}$ holds if 
$$
    2r - 2 - \delta j_2 - j_1^{(1)} - \sum_{i=2}^{2r} (i-1) j_1^{(i)} < -2 (2t+1) \varepsilon.
$$

Observe that $2r - 2$ counts the number of unfixed edges in $G_D$ with which we started (the only two fixed edges are $x,y$).  However, $j_2$ counts both $x,y$ in the number of edges contributed by the $j_2$ tiles, so for our particular choice of $\delta$, the inequality holds.  Hence, condition \ref{D4} is satisfied.  Thus, $\cD$ is a $(d, O(1), \varepsilon)$-simply-bounded conflict system for $\cH$ for all $\varepsilon \in \left(0, \frac{1}{2t+2} \right)$.

By \cref{thm:blackbox}, there exists a $P$-perfect matching $\cM \subseteq \cH$ which contains none of the conflicts from $\cC \cup \cD$, and by \cref{claim: matching is a good coloring}, $\cM$ corresponds to a well-defined edge-coloring of all the edges of $K_{n,n}$, using $|N_1| + |N_2|$ colors in which every copy of $K_{2,t}$ has an odd color class.

Hence, the upper bound in \cref{thm: main_graph} holds.

\section{Proof of Theorem \ref{thm: main_hypergraph}}\label{sec:hyp}

The lower bound in Theorem~\ref{thm: main_hypergraph} follows the same logic as the lower bound in \cref{thm: main_graph} by restricting our attention to the common neighborhood of some fixed $k-2$ vertices, one from each of the first $k-2$ partite sets, and then repeating the argument above with the remaining two partite sets.

Let $N_1$ and $N_2$ be disjoint sets of $\lceil \frac{n}{2}\rceil$ and $\lceil n^\delta \rceil$ colors, respectively,  where $\delta<1$ will be determined later in the proof. To prove the upper bound in Theorem \ref{thm: main_hypergraph}, we will construct an edge-coloring of $\cK_{n, \dots ,n}^{(k)}$ with $|N_1|+|N_2|$ colors such that any copy of $\cK_{1,\dots,1,2,2}$ has an odd color class. Let $\cK_{n,\dots, n}^{(k)}$ have parts $X_1, X_2, \dots, X_k$.

For $i\in [k]$, let $A_i$ be the collection of $i$-vertex sets (or an $i$-tuple) such that in any $i$-tuple, there is at most one vertex from each $X_j$ for each $j\in [k]$. We now define a hypergraph $\mathcal{H}_1$ with vertex set $A\cup B$, where $A = \left\{a \cup \{u\}: a\in A_{k-1}, u\in V(\cK_{n, \dots ,n}^{(k)})\setminus a, \right\}$ and $B = \left\{\{i\} \cup a : i\in N_1, a\in A_{k-1} \right\}$. Note that $A$ includes all the $k$-tuples in $A_k$ when $u$ is in a different vertex side from all the vertices in $a\in A_{k-1}$.

For every copy $S$ of $\cK_{k+1,\dots,k+1}^{(k)}$ in $\cK_{n,\dots,n}^{(k)}$, where $S\cap X_j = \left\{x_1^{(j)},\dots,x_{k+1}^{(j)}\right\}$ for all $j\in [k]$, and for each color $i\in N_1$ we include the hyperedge $e_{S,i}$ in $\mathcal{H}_1$ where 
\begin{align*}
    e_{S,i} = &\{a\in A : \text{$a \subseteq V(S)$ and $a$ is a transversal of $S$}\} \\
    \cup &\{b\in B: b\subseteq V(S)\cup \{i\} \text{ and $b\setminus \{i\}$ is a transversal of $S$}\}.
\end{align*}

Note that a set of vertices $\left\{x^{(j_1)}_{\ell_1},\dots,x^{(j_t)}_{\ell_t} \right\}$ is a transversal if and only if all the $\ell_j$'s are distinct. 

\begin{figure}[h!]
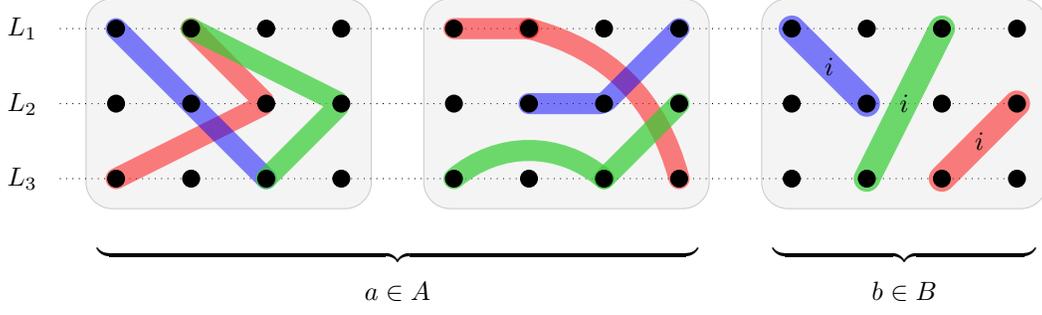

    \centering
    \includestandalone[width=0.9\textwidth]{tiles-example}
    \caption{An example when $k=3$ of the three different types of vertices in $\mathcal{H}_1$. For each type, three specific vertices are given (red, blue, and green). Let $L_i$ denote the corresponding vertices in the vertex side $X_i$ of $\cK_{n,n,...,n}^{(k)}$.}
    \label{fig: ex-tile-thm2}
\end{figure}

Observe that $\cH_1$ is a $\left( (k+1)! + k \binom{k+1}{2} (k-1)(k-1)! + k\frac{(k+1)!}{2} \right)$-uniform hypergraph. To see this, observe that one can choose edges in $A_k$ in $(k+1)!$ ways. To choose the edges in $A \setminus A_{k}$, there are $k$ ways to choose the vertex side containing two vertices, $k-1$ ways to choose the vertex side containing no vertices, $\binom{k+1}{2}$ ways to choose the two vertices in this part, and $(k-1)!$ ways to choose the rest of the vertices. Lastly, to choose the vertices in $B$, one can choose the part that is omitted in $k$ ways, and then choose the remaining vertices in $\frac{(k+1)!}{2}$ ways. 

For the rest of the proof, we refer to the edges of $\cH_1$ as \emph{tiles}.  Hence, when we refer to an \emph{edge}, we are referring to an edge of the hypergraph $\cK_{n, \dots, n}^{(k)}$.  When we refer to a \emph{tile}, we are referring to an edge of the hypergraph $\cH_1$.  For clarity, to match the setup for the conflict-free hypergraph matching method, define $P := A$ and $Q := B$.  Let $\varepsilon_0 = \frac{1}{k^2+2}$, and define 
$$
    \max \left\{ 1 - \frac{k^2 + 1}{2(k^2 + 2)^4}, \frac{k^2 + 1}{k^2 + 2} , \frac{1}{2} \left(1 + \frac{k^2+1}{k^2+2}\right) \right\} < \delta  < 1.
$$

Observe that such a $\delta$ does indeed exist.

\begin{claim}
The hypergraph $\cH_1$ is essentially $d$-regular for $d = \frac{n^{k^2+1}}{2}$, thus satisfying \ref{H1}.
\end{claim}

\begin{proof}
We check that for any vertex $v \in V(\cH_1)$, we have $d - O(n^{k^2+1}) \leq \deg_{\cH_1}(v) \leq d$.  

First, suppose $v \in A$ and $v$ is of the form $\left\{x_{\ell_1}^{(1)},x_{\ell_2}^{(2)},\dots,x_{\ell_k}^{(k)}\right\}$ (see left of \cref{fig: ex-tile-thm2}).  Then there are $\frac{n^{k^2+1}}{2} - O(n^{k^2})$ many tiles in $\cH_1$ containing $a$. Indeed, to choose another transversal edge that does not intersect $v$, there are $n-1$ ways to choose each of the $k$ vertices in that transversal. Continuing to choose disjoint transversal edges in this way, there are $n-j$ ways to choose each vertex in the $j$-th disjoint transversal. Since there are $k$ disjoint transversals, there are $n^{k^2} - O(n^{k^2 - 1})$ ways to choose all disjoint transversal edges, which now define the remaining edges of $S$. Lastly, we choose a color in $\frac{n}{2}$ ways, which gives the desired bound.

Next, suppose $v \in A$ and $v$ is of the form $v=\left\{x^{(j_1)}_{\ell_1},x^{(j_1)}_{\ell_2}, x^{(j_2)}_{\ell_3}, \dots, x^{(j_{k-1})}_{\ell_{k}}\right\}$ (see middle of \cref{fig: ex-tile-thm2}), then there are $(n-2)^{k-1} - O(n^{k-2})$ ways to pick the remaining vertices from $X_{j_1}$, $(n-1)^{k(k-2)} - O(n^{k(k-2)-1})$ ways to pick the remaining vertices from parts $j_2, \dots, j_{k-1}$, $n^{k+1} - O(n^k)$ ways to choose the remaining vertices from the remaining part, and $\frac{n}{2}$ ways to choose the color of the tile.  

Lastly, if $v\in B$ is of the form $v = \left\{x^{(j_1)}_{\ell_1},  x^{(j_2)}_{\ell_2}, \dots, x^{(j_{k-1})}_{\ell_{k-1}}, i \right\}$ (see right of \cref{fig: ex-tile-thm2}), then there are $\binom{n-1}{2}$ ways to choose the two vertices which each complete a transversal with $x^{(j_1)}_{\ell_1},  x^{(j_2)}_{\ell_2}, \dots, x^{(j_{k-1})}_{\ell_{k-1}}$, and $\left( (n-1)^{k(k-1)} - O(n^{k(k-1)-1}) \right) \cdot \left((n-2)^{k-1} - O(n^{k-2}) \right) = n^{k^2 -1} - O(n^{k^2-2})$ ways to choose the remaining vertices. 

Thus, in each case, we have $\deg_{\cH_1}(v) = \frac{1}{2} n^{k^2 + 1} - O(n^{k^2})$.
\end{proof}

Note that for $\varepsilon = \frac{1}{k^2 + 2}$ and $n$ sufficiently large, 
$$
    d^\varepsilon
    = O \left( n^{\frac{k^2 + 1}{k^2 + 2}} \right)
    \leq O(n^k) = |P| \leq |P \cup Q| 
    \leq O(n^k)
    \leq \exp \left( c_k n^{\frac{k^2 + 1}{(k^2 + 2)^3}} \right)
    = \exp(d^{\varepsilon^3}),
$$
where $c_k = 2^{\frac{-1}{(k^2 + 2)^3}}$.

Furthermore, we have $\Delta_2(\cH_1) \leq d^{1-\varepsilon} $ for all $\varepsilon \in \left( 0,\frac{1}{k^2+2} \right)$, since each pair of vertices $u,v \in V(\cH_1)$ is contained in at most $O(n^{k^2})$ tiles in $\cH_1$.  To see this, given, $u,v \in V(\cH_1)$, if $u$ and $v$ appear in a tile $e_{S,i}$ of $\cH_1$ together, then either at least $k+1$ vertices of a copy of $S$ are given, or at least $k$ vertices and a color of the copy of $S$ are given.  In the first case, there are at most $O(n^{k^2})$ ways to choose the remaining $k^2 - 1$ vertices and a color of $S$.  In the second case, there are at most $O(n^{k^2})$ ways to choose the remaining $k^2$ vertices of $S$.  Hence, there are at most $O(n^{k^2})$ tiles containing both $u$ and $v$, satisfying \ref{H2}.  

Let $\cH_2$ be the $(k+1)$-uniform hypergraph with vertex set $A \cup C$, where $C= \left\{a\cup \{i\}: a\in A_{k-1}, i\in N_2 \right\}$ and for each $a\in A_k$ and $i\in N_2$ we define the edge

\[e_{a,i}=\{a\}\cup \{a'\cup \{i\}:a'\in A_{k-1}, a'\subset a\}.\]

That is, $e_{i}$ consists of one vertex from $A_k$ and $k$ vertices from $C$, and corresponds to coloring one edge in $\cK^{(k)}_{n, \dots, n}$.  Define $R := C$ and $\cH := \cH_1 \cup \cH_2$.  

We have $\Delta_R(\cH_2) \leq d^{\varepsilon^4} \delta_P(\cH_2)$, since $\Delta_R(\cH_2) = n$ (vertices in $R$ are $k$-tuples consisting of $k-1$ vertices of $\cK^{(k)}_{n, \dots,n}$ and a color from $N_2$, so there are $n$ ways to choose the final vertex in the edge of $\cK^{(k)}_{n, \dots,n}$ to be colored) and $\delta_P(\cH_2) = \left\lceil n^{\delta} \right\rceil$ (since there are $\left\lceil n^\delta \right \rceil$ ways to choose $i$).  One can verify that $n < O(n^{(k^2 + 1) \varepsilon^4 + \delta})$. Hence, \ref{H3} is satisfied.  

Furthermore, for any $T \in P$ and $S \in R$, we have $d(T,S) \leq d^{-\varepsilon} n^{\delta}$ since $d(T,S) = 0$ if $T$ is not an edge in $\cK^{(k)}_{n,\dots ,n}$ (and hence not in any tile of $\cH$ along with $S$), and $d(T,S) = 1$ if $T$ is an edge in $\cK^{(k)}_{n,\dots ,n}$ (since the edge of $\cK^{(k)}_{n, \dots ,n}$ and its color are both fixed by $S,T$).  Hence, \ref{H4} is satisfied for our particular choice of $\delta$.  

Next, we define a conflict system $\cC$ for $\cH_1$ with edges of sizes 3 and 4.

\subsection{The conflict system $\cC$} \label{sec:hypC}

    Let $V(\cC) = E(\cH_1)$, and let the edges of $\cC$ correspond to copies of $\cK_{1, \dots, 1,2,2}$ in $\cK_{n, \dots ,n}^{(k)}$ colored as in \cref{fig:bad_config}.  Note that the graph in \cref{fig:bad_config} is colored in the only way that is allowed by a matching of $\cH_1$, yet still having no odd color class.  We will call these edges of $\cC$ \emph{conflicts}, and they correspond to sets of three or four monochromatic copies $S$, $S'$, $T$, and $T'$ which are each the transversals of a copy of $\cK_{k+1,\dots,k+1}^{(k)}$ in $\cK_{n,\dots,n}^{(k)}$.  In particular, we have $\{e_{S,i},e_{S',i}, e_{T,j}, e_{T',j}\} \in E(\cC)$ for each set of vertices $x_1 \in X_{j_1}, \dots,$ $x_{k-2} \in X_{j_{k-2}}$, $x_{k-1}, x'_{k-1} \in X_{j_{k-1}}$, and $x_{k}, x'_{k} \in X_{j_{k}}$, distinct colors $i,j \in N_1$, and (not necessarily distinct) tiles $e_{S,i}, e_{S',i}, e_{T,j}, e_{T',j} \in E(\cH_1)$ with 
    \begin{align*}
    \left\{x_1, \dots, x_{k-2}, x_{k-1}, x_k\right\} &\in e_{S,i},\\
    \left\{x_1, \dots, x_{k-2}, x'_{k-1}, x'_k\right\} &\in e_{S',i},\\
    \left\{x_1, \dots, x_{k-2}, x'_{k-1}, x_k\right\} &\in e_{T,j},  \\
    \text{and }\left\{x_1, \dots, x_{k-2}, x_{k-1}, x'_k\right\} &\in e_{T',j}.
    \end{align*}
    Note that we do not need to consider conflicts of size two: if we did, then $e_{S,i} = e_{S',i}$ and $e_{T,j} = e_{T',j}$, but then $\{x_1, \dots, x_{k-2}, x_{k-1}, x'_{k-1}\}$ appears in both $e_{S,i}$ and $e_{T,j}$, and thus $e_{S,i}$ and $e_{T,j}$ cannot be in a matching of $\cH_1$.  Hence, \ref{C1} is satisfied.

    \begin{figure}[h!]
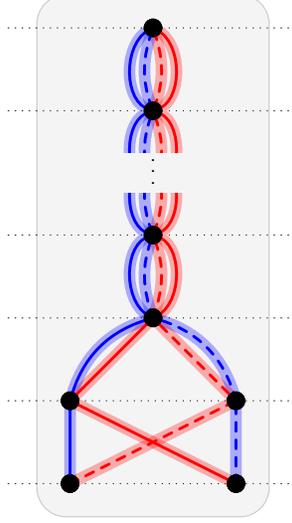

        \centering
        \includestandalone[width = .25\textwidth]{hypergraph-conflict}
        \caption{An example of the only bad coloring of $\cK_{1, \ldots ,1, 2, 2}$ in $\cK_{n,\dots,n}^{(k)}$ whose hyperedges come from tiles which could appear in a matching of $\cH$. }
        \label{fig:bad_config}
    \end{figure}

    We will now show that $\Delta(\cC^{(3)}) = O(d^2)$ and $\Delta(\cC^{(4)}) = O(d^3)$.  To this end, fix $e_{S,i} \in V(\cC)$.  Without loss of generality, any conflict in $\cC^{(3)}$ has the form $\{e_{S,i}, e_{T,j}, e_{T',j}\}$ or $\{e_{S,i}, e_{S',i}, e_{T,j}\}$.  If the conflict has the form $\{e_{S,i}, e_{T,j}, e_{T',j}\}$, then there are $O( d)$ ways to choose $e_{T,j}$ and $O(d)$ ways to choose $e_{T',j}$.  If the conflict has the form $\{e_{S,i}, e_{S',i}, e_{T,j}\}$, we claim there are $O(n \cdot d)$ ways to choose $e_{S',i}$. Note that there are $k-2$ vertices in an edge of $S'$, and we choose another vertex in $n$ ways. Now, we know a $k$-tuple in $e_{S',i}$, which consists of the $k-1$ vertices along with color $i$, so we can choose $e_{S',i}$ in $d$ ways. Finally, we can choose $e_{T,j}$ in $O(\frac{d}{n})$ ways. Indeed, since we know at least two $k$-tuples now in $e_{T,j}$, we can choose $e_{T,j}$ in $\Delta_2 (\cH_1) = O(\frac{d}{n})$ ways.  In either case, $\deg_\cC(e_{S,i}) = O(d^2)$, proving $\Delta(\cC^{(3)}) = O(d^2)$.  Similarly, any conflict in $\cC^{(4)}$ containing $e_{S,i}$ must have the form $\{e_{S,i}, e_{S',i}, e_{T,j}, e_{T',j}\}$.  Hence, there are $O(n \cdot d)$ ways to choose $e_{S,i}$ (since there are $O(n)$ ways to choose a vertex from $\cK_{n,\dots,n}^{(k)}$ to be in a $k$-tuple with color $i$ in $e_{S',i}$, and then we know a vertex of the form $\{x_1, x_2, \dots, x_{k-1},i\}\in e_{S',i}$ and $\deg_{\cH_1}(e_{S',i}) = O(d)$), and there are $O(d)$ ways to choose $e_{T,j}$ and $O \left(\frac{d}{n} \right)$ ways to choose $e_{T',j}$.  Thus, $\deg_\cC(e_{S,i}) = O(d^3)$, proving $\Delta(\cC^{(4)}) = O(d^3)$.  Hence, \ref{C2} is satisfied.

    Next, we show $\Delta_{j'}(\cC^{(j)}) \leq d^{j-j'-\varepsilon}$ for all $j \in [3,4]$ and $j' \in [2,j-1]$ for all $\varepsilon \in \left( 0,\frac{1}{k^2+2} \right)$. Indeed, the number of conflicts in $\cC^{(3)}$ containing two fixed tiles is $O(d^{1-\varepsilon})$.  To see this, a conflict in $\cC^{(3)}$ has the form $\{e_{S,i}, e_{S',i}, e_{T,j}\}$.  If $e_{S,i}, e_{S',i}$ are the fixed tiles, then we know two $k$-tuples of the form $\{x_1,\dots,x_k\}$ and $\{x'_1,\dots,x'_k\}$ in $e_{T,j}$.  Hence, there are at most $\Delta_2(\cH_1) = O \left( \frac{d}{n} \right) \leq d^{1-\varepsilon}$ ways to choose $e_{T,j}$ to be in the conflict.  If, without loss of generality, $e_{S,i}$ and $e_{T,j}$ are fixed, then we know two $k$-tuples in $e_{S',i}$, and there are $O\left(\frac{d}{n} \right)$ ways to choose $e_{S',i}$ to be in the conflict.  

    Similarly, the number of conflicts in $\cC^{(4)}$ containing two fixed vertices is $O(d^{2-\varepsilon})$.  Observe that a conflict in $\cC^{(4)}$ has the form $\{e_{S,i}, e_{S',i}, e_{T,j}, e_{T',j}\}$.  If $e_{S,i}$ and $e_{S',i}$ are fixed, then there are $O(d)$ ways to choose $e_{T,j}$ and $O\left(\frac{d}{n}\right)$ ways to choose $e_{T',j}$.  If $e_{S,i}$ and $e_{T,j}$ are fixed, then there are $O(d)$ ways to choose $e_{S',i}$ (since we know a $k$-tuple of the form $\left\{x_1, \dots, x_{k-1},i\right\}$ in $e_{S',i}$), and $O\left(\frac{d}{n}\right)$ ways to choose $e_{T',j}$ (since we know two $k$-tuples, one of the form $\left\{x_1,\dots, x_k\right\}$ and one of the form $\left\{x_1, \dots, x_{k-1},j\right\}$, in $e_{T',j}$).  Lastly, the number of conflicts in $\cC^{(4)}$ containing three fixed vertices is $O(d^{1-\varepsilon})$.  Without loss of generality, suppose $e_{S,i}$, $e_{S',i}$, and $e_{T,j}$ are fixed.  Then there are $O\left(\frac{d}{n}\right)$ ways to choose $e_{T',j}$ to form the conflict since we know two $k$-tuples, of the form $\{x_1, \dots, x_k\}$ and $\{x_1, \dots, x_{k-1}, j\}$, in $e_{T',j}$.  Hence, \ref{C3} is satisfied.  

    Thus, $\cC$ is a $(d,O(1), \varepsilon)$-bounded conflict system for $\cH_1$ for all $\varepsilon \in \left( 0, \frac{1}{k^2+2} \right)$.

\subsection{The conflict system $\cD$} \label{sec:hypD}

    Let $V(\cD) = E(\cH)$, and let the edges of $\cD$ correspond to copies of $\cK_{1,\dots, 1,2,2}$ in $\cK_{n,\dots,n}^{(k)}$ colored as in \cref{fig:bad_config}.  Note that the coloring of the graph in \cref{fig:bad_config} is again the only coloring of $\cK_{1,\dots,1,2,2}$ which has no odd color class since we are using a matching of $\cH = \cH_1 \cup \cH_2$ to color $\cK^{(k)}_{n,\dots, n}$.  To see this, observe that any two edges corresponding to two tiles in $\cH_2$ cannot overlap in $k-1$ graph vertices, since then those edges would not be in a matching of $\cH_2$ (and thus not in a matching of $\cH$).  Again, we will call these edges of $\cD$ \emph{conflicts}, and without loss of generality, they correspond to sets of one or two monochromatic copies $S, S'$ which are each transversals of a copy of $\cK_{k+1,\dots,k+1}^{(k)}$ in $\cK_{n,\dots,n}^{(k)}$, and two or four copies $T_1, T_2, T_3, T_4$ which are each edges in $\cK_{n,\dots,n}^{(k)}$. A conflict of the first type with only two tiles coming from $\cH_2$ we call \emph{type 1}, and a conflict of the second type with four tiles coming from $\cH_2$ we call \emph{type 2}.  Note that type 1 conflicts can be of size three or four, while type 2 conflicts are of size four. See \cref{fig:bad_config2} for an example.

    \begin{figure}[h!]
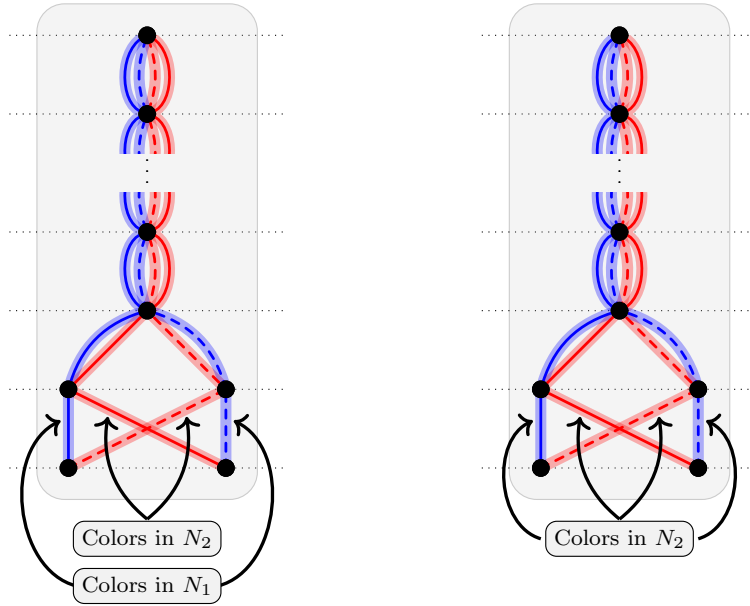

        \centering

    \includestandalone[width = .65\textwidth]{hypergraph-conflict-D}
        \caption{The bad colorings of $\cK_{1, \dots, 1, 2, 2}$ in $\cK_{n,n,...,n}^{(k)}$ which correspond with the different types of conflicts in $\mathcal{D}$ (type 1 on the left and type 2 on the right). Note that two edges of the same color from $N_1$ can come from the same tile in $\cH_1$.}
        \label{fig:bad_config2}
    \end{figure}

    Our above definition of the conflicts in $\cD$ implies that any conflict has at least two, and no more than four, tiles in $\cH_2$.  Hence, the property \ref{D1} is satisfied.  

    We next show $\left| \cD_x^{(j_1,j_2)} \right| \leq d^{j_1 + \varepsilon^4} \delta_P(\cH_2)^{j_2}$ for all $j_1 \in [0,O(1)], j_2 \in [2,O(1)]$.  To this end, first, observe that $\delta_P(\cH_2) = n^\delta$.  To see this, a vertex in $P$ that is also in $V(\cH_2)$ is an edge $\{x_1, x_2, \dots, x_k\}$ of $\cK_{n,\dots,n}^{(k)}$, and there are $n^\delta$ ways to choose a color for the edge; tiles in $\cH_2$ correspond to colored edges, so the observation follows.  Now, consider $\cD_x^{(1,2)}$, which contains conflicts only of type 1.  If $x \in P$ is in some tile in $\cH_2$ that is in a conflict in $\cD$, then there are $\frac{n}{2}$ ways to choose the color of the tile from $\cH_1$, then $d$ ways to choose that tile, and $n^\delta$ ways to choose the color of $x$ and the other tile in the conflict.  Hence, $\left| \cD_x^{(1,2)} \right| \leq O(n^{1+\delta} d) \leq d^{1+\varepsilon^4} n^{2\delta}$ by our choice of $\varepsilon$ and $\delta$ (since $1 - \frac{k^2 + 1}{(k^2 + 2)^4} < \delta)$.  Next, consider $\cD_x^{(2,2)}$, which again contains conflicts only of type 1.  Again, we fix $x \in P$ in some tile in $\cH_2$ that is in a conflict in $\cD$.  There are $O(n)$ ways to choose the color of the two tiles coming from $\cH_1$, then $d$ ways to choose the first tile coming from $\cH_1$, then $d$ ways to choose the second tile coming from $\cH_1$, and $n^\delta$ ways to choose the color of the tiles from $\cH_2$ (all edges in the tiles from $\cH_2$ are fixed by the tiles from $\cH_1$).  Hence, $\left| \cD_x^{(2,2)} \right| \leq O(n^{1+\delta} d^2) \leq d^{2+\varepsilon^4} n^{2\delta}$ by our choice of $\varepsilon$ and $\delta$ (since $1 - \frac{k^2 + 1}{(k^2 + 2)^4} < \delta)$.  Now, consider $\cD_x^{(0,4)}$, which contains only conflicts of type 2.  Note there are $O(n^2)$ ways to choose the two remaining graph vertices and $(n^{\delta})^2$ ways to choose the two colors of the tiles.  Hence, $\left| \cD_x^{(0,4)} \right| \leq O(n^{2 + 2\delta}) \leq d^{\varepsilon^4} n^{4\delta}$ by our choice of $\varepsilon$ and $\delta$ (since $1 - \frac{k^2 + 1}{2(k^2 + 2)^4} < \delta)$.  Finally, note that for all other $j_1, j_2$, $\cD_x^{(j_1,j_2)}$ is empty and hence the bound holds trivially.  Thus, \ref{D2} is satisfied.  

    Next, we will show $\Delta_{j',0} \left(\cD_x^{(j_1,j_2)} \right) \leq d^{j_1 - j' - \varepsilon} \delta_P(\cH_2)^{j_2}$ for each $0 \leq j' \leq j_1$.  Note that in our case, as we observed above, $j_1 \in \{ 0, 1, 2 \}$.  If $j_1 = 0$, then $j' = 0$, and $\Delta_{0,0}\left(\cD_x^{(0,4)} \right)$ = 0 since in $\cD_x^{(0,4)}$, there are no conflicts that contain zero tiles from $\cH_1$ and zero tiles from $\cH_2$.  If $j_1 = 1$, then $j_2 = 2$ and $j' = 0$ or $j' = 1$.  Note that again we have $\Delta_{0,0} \left(\cD_x^{(1,2)} \right) = 0$.  We also have $\Delta_{1,0} \left(\cD_x^{(1,2)} \right) = O(n^\delta)$ (since for a fixed tile from $\cH_1$ in a conflict in $\cD_x^{(1,2)}$, there are $n^\delta$ ways to choose the color of the tiles from $\cH_2$ in the conflict, but the graph edges in the tile are already determined by the fixed tile from $\cH_1$).  Hence, $\Delta_{1,0}\left(\cD_x^{(1,2)}\right) = O(n^\delta) \leq d^{-\varepsilon} n^{2\delta}$ for our particular choice of $\varepsilon$ and $\delta$ (since $\frac{k^2+1}{k^2+2} < \delta$).  If $j_1 = 2$, then $j_2 = 2$ again, and $j' \in \{0,1,2\}$.  We have $\Delta_{0,0}\left(\cD_x^{(2,2)}\right) = 0$.  Similarly, we have $\Delta_{1,0} \left(\cD_x^{(2,2)} \right) = O(d \cdot n^\delta)$ since for a given tile in $\cD_x^{(2,2)}$ from $\cH_1$, it must appear in a conflict with an $\cH_2$ tile containing $x$, so there are $d$ ways to choose the other $\cH_1$ tile and $n^\delta$ ways to choose the color of the $\cH_2$ tiles (the graph edges corresponding to the $\cH_2$ tiles are fixed by the $\cH_1$ tiles and $x \in P$).  Hence, $\Delta_{1,0}\left(\cD_x^{(2,2)}\right) = O(d \cdot n^\delta) \leq d^{1 -\varepsilon} n^{2\delta}$ for our particular choice of $\varepsilon$ and $\delta$ (since $\frac{k^2+1}{k^2+2} < \delta$).  Lastly, we have $\Delta_{2,0}\left(\cD_x^{(2,2)}\right) = O(n^\delta)$ since for two given tiles in $\cD_x^{(2,2)}$ from $\cH_1$ to appear together in a conflict, there are $n^\delta$ ways to choose the color of the $\cH_2$ tiles (the graph edges corresponding to the $\cH_2$ tiles are fixed by the $\cH_1$ tiles).  Hence, $\Delta_{2,0}\left(\cD_x^{(2,2)}\right) = O( n^\delta) \leq d^{-\varepsilon} n^{2\delta}$ for our particular choice of $\varepsilon$ and $\delta$ (since $\frac{k^2+1}{k^2+2} < \delta$).  Thus, \ref{D3} is satisfied.  

    Finally, we will show $\left|\cD_{x,y}^{(j_1,j_2)}\right| \leq d^{j_1 - \varepsilon} \delta_P(\cH_2)^{j_2}$ for $x,y \in P$ in the $\cH_2$ part of the conflicts from $\cD$.  First observe that $\left|\cD_{x,y}^{(0,4)}\right| = O(n^{1+2\delta})$ since for any fixed $x,y$, there are $O(n)$ ways to choose the last graph vertex in a graph edge that corresponds to the other two $\cH_2$ tiles in a conflict, and $n^{2\delta}$ ways to choose the colors of the tiles.  Hence, $\left|\cD_{x,y}^{(0,4)}\right| = O(n^{1+2\delta}) \leq d^{-\varepsilon} n^{4\delta}$ since $\frac{1}{2} (1 + \frac{k^2+1}{k^2+2}) < \delta$.  Next observe that $\left|\cD_{x,y}^{(1,2)}\right| = O\left( \frac{d}{n} \cdot n^{\delta}\right)$ since for any fixed $x,y$, there are $O\left(\frac{d}{n}\right)$ ways to choose the tile from $\cH_1$ (we know two graph edges in the tile, they are fixed by $x,y$), and $n^{\delta}$ ways to choose the color of the $\cH_2$ tiles.  Hence, $\left|\cD_{x,y}^{(1,2)}\right| = O\left(\frac{d}{n}\right) \leq d^{1-\varepsilon} n^{2\delta}$ since $\frac{1}{2} \left( \frac{k^2+1}{k^2+2} - 1 \right) < \delta$.  Lastly, observe that $\left|\cD_{x,y}^{(2,2)}\right| = O\left( \frac{d^2}{n}  \cdot n^{\delta}\right)$ since for any fixed $x,y$, there are $O\left(d \cdot \frac{d}{n}\right)$ ways to choose the tiles from $\cH_1$ (we know a graph edge in one tile, it is fixed by $x,y$, and after choosing one tile from $\cH_1$, we know a graph edge and a color of the last tile), and $n^{\delta}$ ways to choose the color of the $\cH_2$ tiles.  Hence, $\left|\cD_{x,y}^{(2,2)}\right| = O\left(\frac{d^2}{n} \cdot n^\delta \right) \leq d^{2-\varepsilon} n^{2\delta}$ since $\frac{k^2+1}{k^2+2} - 1 < \delta$.  Thus, \ref{D4} is satisfied.  Thus, $\cD$ is a $(d,O(1),\varepsilon)$-simply-bounded conflict system for $\cH$ for all $\varepsilon \in \left(0,\frac{1}{k^2 + 2} \right)$.

    By \cref{thm:blackbox}, there exists a $P$-perfect matching $\cM \subseteq \cH$ which contains none of the conflicts from $\cC \cup \cD$.  This matching corresponds to a well-defined edge-coloring of all the edges of $\cK^{(k)}_{n,\dots,n}$, using $N_1+N_2=\lceil\frac{n}{2}\rceil + \lceil n^\delta\rceil$ colors, in which every copy of $\cK_{1,\dots,1,2,2}$ has an odd color class.  Hence, the upper bound in \cref{thm: main_hypergraph} holds.

\section{Acknowledgments}
This project began during the 2024 Graduate Research Workshop in Combinatorics, which was supported by the University of Wisconsin-Milwaukee, the Combinatorics Foundation, and the National Science Foundation (NSF Grant DMS-1953445). 
We would like to thank Ilkyoo Choi, The Nguyen, Florian Pfender, Youngho Yoo, and Yufei Zhang for helpful conversations at the start of the project.

\bibliography{bibfile}
\bibliographystyle{abbrv}

\end{document}

%% file: reducible-conflict.tex
\begin{tikzpicture}

\draw[very thick, orange] (0,1) -- (2,3);
\draw[very thick, cyan] (0,1) -- (2,2);
\draw[very thick, red] (0,1) -- (2,1);
\draw[very thick, red] (0,1) -- (2,0);
\draw[very thick, yellow] (0,1) -- (2,-1);
\draw[very thick, green] (0,1) -- (2,-2);

\draw[very thick, cyan] (0,0) -- (2,3);
\draw[very thick, orange] (0,0) -- (2,2);
\draw[very thick, blue] (0,0) -- (2,1);
\draw[very thick, green] (0,0) -- (2,0);
\draw[very thick, blue] (0,0) -- (2,-1);
\draw[very thick, yellow] (0,0) -- (2,-2);

\draw[fill = black] (0,0) circle (3pt);
\draw[fill = black] (0,1) circle (3pt);
\draw[fill = black] (2,3) circle (3pt);
\draw[fill = black] (2,2) circle (3pt);
\draw[fill = black] (2,1) circle (3pt);
\draw[fill = black] (2,0) circle (3pt);
\draw[fill = black] (2,-1) circle (3pt);
\draw[fill = black] (2,-2) circle (3pt);

\draw[very thick, orange] (4,3) -- (6,3);
\draw[very thick, cyan] (4,3) -- (6,2);
\draw[very thick, cyan] (4,2) -- (6,3);
\draw[very thick, orange] (4,2) -- (6,2);

\draw[fill = black] (4,3) circle (3pt);
\draw[fill = black] (4,2) circle (3pt);

\draw[fill = black] (6,3) circle (3pt);
\draw[fill = black] (6,2) circle (3pt);

\draw[dashed] (4,1.5) -- (6,1.5);

\draw[very thick, red] (4,-.5+.5) -- (6,.5+.5);
\draw[very thick, blue] (4,-1.5+.5) -- (6,.5+.5);
\draw[very thick, red] (4,-.5+.5) -- (6,-.5+.5);
\draw[very thick, green] (4,-1.5+.5) -- (6,-.5+.5);
\draw[very thick, yellow] (4,-.5+.5) -- (6,-1.5+.5);
\draw[very thick, blue] (4,-1.5+.5) -- (6,-1.5+.5);
\draw[very thick, green] (4,-.5+.5) -- (6,-2.5+.5);
\draw[very thick, yellow] (4,-1.5+.5) -- (6,-2.5+.5);

\draw[fill = black] (4,-.5+.5) circle (3pt);
\draw[fill = black] (4,-1.5+.5) circle (3pt);

\draw[fill = black] (6,-2.5+.5) circle (3pt);
\draw[fill = black] (6,-1.5+.5) circle (3pt);
\draw[fill = black] (6,-.5+.5) circle (3pt);
\draw[fill = black] (6,.5+.5) circle (3pt);

\end{tikzpicture}

%% file: leaf-combos.tex
\begin{tikzpicture}

\begin{scope}[shift = {(0,0)}]

\draw[thick, gray!70!] (0,.5) to (2,1.6);
\draw[thick, gray!70!] (0,.5) to (2,1.2);
\draw[thick, gray!70!] (0,.5) to (2,.8);
\draw[thick, gray!70!] (0,.5) to (2,.4);
\draw[thick, gray!70!] (0,.5) to (2,0);
\draw[thick, gray!70!] (0,.5) to (2,-.4);

\draw[thick, gray!70!] (0,-.5) to (2,.4);
\draw[thick, gray!70!] (0,-.5) to (2,0);
\draw[thick, gray!70!] (0,-.5) to (2,-.4);
\draw[thick, gray!70!] (0,-.5) to (2,-.8);
\draw[thick, gray!70!] (0,-.5) to (2,-1.2);

\draw[line width = 1.5pt, black, rounded corners = 2pt] (0,-.5) to (2,1.6);
\draw[line width = 1.5pt, black, rounded corners = 2pt] (0,-.5) to (2,1.2);
\draw[line width = 1.5pt, black, rounded corners = 2pt] (0,-.5) to (2,.8);
\draw[line width = 1.5pt, black, rounded corners = 2pt] (0,-.5) to (2,-1.6);

\draw[fill = black] (0,-.5) circle (2pt);
\draw[fill = black] (0,.5) circle (2pt);

\draw[fill = black] (2,1.6) circle (2pt);

\draw[fill = black] (2,1.2) circle (2pt);
\draw[fill = black] (2,.8) circle (2pt);
\draw[fill = black] (2,.4) circle (2pt);
\draw[fill = black] (2,0) circle (2pt);
\draw[fill = black] (2,-.4) circle (2pt);
\draw[fill = black] (2,-.8) circle (2pt);
\draw[fill = black] (2,-1.2) circle (2pt);
\draw[fill = black] (2,-1.6) circle (2pt);

\draw[draw=red, fill=none] (2,1.6) circle (3pt);
\draw[draw=red, fill=none] (2,1.2) circle (3pt);
\draw[draw=red, fill=none] (2,.8) circle (3pt);

\draw[draw=green, fill=none] (2,-1.6) circle (3pt);

\node[red] at (2.35,1.6) {\tiny{$-1$}};
\node[red] at (2.35,1.2) {\tiny{$-1$}};
\node[red] at (2.35,.8) {\tiny{$-1$}};
\node[green] at (2.35,-1.6) {\tiny{$+1$}};

\node[font=\scriptsize] at (1,-2.2) {$\ell_i = \ell_{i-1}-2$};

\end{scope}

\begin{scope}[shift = {(3.5,0)}]

\draw[thick, gray!70!] (0,.5) to (2,1.6);
\draw[thick, gray!70!] (0,.5) to (2,1.2);
\draw[thick, gray!70!] (0,.5) to (2,.8);
\draw[thick, gray!70!] (0,.5) to (2,.4);
\draw[thick, gray!70!] (0,.5) to (2,0);
\draw[thick, gray!70!] (0,.5) to (2,-.4);

\draw[thick, gray!70!] (0,-.5) to (2,.4);
\draw[thick, gray!70!] (0,-.5) to (2,0);
\draw[thick, gray!70!] (0,-.5) to (2,-.4);
\draw[thick, gray!70!] (0,-.5) to (2,-.8);
\draw[thick, gray!70!] (0,-.5) to (2,-1.2);

\draw[line width=1.5pt, black, rounded corners=2pt, dash pattern=on 3pt off 1.25pt] (0,.5) to (2,-1.2);

\draw[line width = 1.5pt, black, rounded corners = 2pt] (0,-.5) to (2,1.6);
\draw[line width = 1.5pt, black, rounded corners = 2pt] (0,-.5) to (2,1.2);
\draw[line width = 1.5pt, black, rounded corners = 2pt] (0,-.5) to (2,-1.6);

\draw[fill = black] (0,-.5) circle (2pt);
\draw[fill = black] (0,.5) circle (2pt);

\draw[fill = black] (2,1.6) circle (2pt);

\draw[fill = black] (2,1.2) circle (2pt);
\draw[fill = black] (2,.8) circle (2pt);
\draw[fill = black] (2,.4) circle (2pt);
\draw[fill = black] (2,0) circle (2pt);
\draw[fill = black] (2,-.4) circle (2pt);
\draw[fill = black] (2,-.8) circle (2pt);
\draw[fill = black] (2,-1.2) circle (2pt);
\draw[fill = black] (2,-1.6) circle (2pt);

\draw[draw=red, fill=none] (2,1.6) circle (3pt);
\draw[draw=red, fill=none] (2,1.2) circle (3pt);
\draw[draw=red, fill=none] (2,-1.2) circle (3pt);

\draw[draw=green, fill=none] (2,-1.6) circle (3pt);

\node[red] at (2.35,1.6) {\tiny{$-1$}};
\node[red] at (2.35,1.2) {\tiny{$-1$}};
\node[red] at (2.35,-1.2) {\tiny{$-1$}};
\node[green] at (2.35,-1.6) {\tiny{$+1$}};

\node[font=\scriptsize] at (1,-2.2) {$\ell_{i+1} = \ell_{i-1}-2$};

\end{scope}

\end{tikzpicture}

%% file: tiles-example.tex
\begin{tikzpicture}

\draw[thin, dotted] (.25,3) to (13,3);
\draw[thin, dotted] (.25,2) to (13,2);
\draw[thin, dotted] (.25,1) to (13,1);

\node at (-.25,3) {$L_1$};
\node at (-.25,2) {$L_2$};
\node at (-.25,1) {$L_3$};

\node at (4.75,0) {$\underbrace{\hspace{8cm}}$};
\node at (4.75,-.5) {$a\in A$};

\begin{scope}[shift = {(0,0)}]
\draw[rounded corners = 10pt, fill = gray!40!, opacity =.2] (1-.4,3+.4) -- (4+.4,3+.4) -- (4+.4,1-.4) -- (1-.4,1-.4) -- cycle;

\draw[rounded corners = 2pt, line width = 8pt, red, line cap = round, opacity = .5] (2,3) -- (3,2) -- (1,1);

\draw[rounded corners = 2pt, line width = 8pt, blue, line cap = round, opacity = .5] (1,3) -- (2,2) -- (3,1);

\draw[rounded corners = 2pt, line width = 8pt, green, line cap = round, opacity = .7] (2,3) -- (4,2) -- (3,1);

\draw[thick, fill = black] (1,3) circle (3pt); 
\draw[thick, fill = black] (2,3) circle (3pt); 
\draw[thick, fill = black] (3,3) circle (3pt); 
\draw[thick, fill = black] (4,3) circle (3pt); 

\draw[thick, fill = black] (1,2) circle (3pt); 
\draw[thick, fill = black] (2,2) circle (3pt); 
\draw[thick, fill = black] (3,2) circle (3pt); 
\draw[thick, fill = black] (4,2) circle (3pt); 

\draw[thick, fill = black] (1,1) circle (3pt); 
\draw[thick, fill = black] (2,1) circle (3pt); 
\draw[thick, fill = black] (3,1) circle (3pt); 
\draw[thick, fill = black] (4,1) circle (3pt);

\end{scope}


\begin{scope}[shift = {(4.5,0)}]

\draw[rounded corners = 10pt, fill = gray!40!, opacity =.2] (1-.4,3+.4) -- (4+.4,3+.4) -- (4+.4,1-.4) -- (1-.4,1-.4) -- cycle;

\draw[rounded corners = 2pt, line width = 8pt, red, line cap = round, opacity =.5] (1,3) -- (2,3);
\draw[rounded corners = 2pt, line width = 8pt, red, line cap = round, bend left = 30, opacity = .5] (2,3) to (4,1);

\draw[rounded corners = 2pt, line width = 8pt, blue, line cap = round, bend left = 30, opacity = .5] (4,3) -- (3,2) -- (2,2);

\draw[rounded corners = 2pt, line width = 8pt, green, line cap = round, bend left = 40, opacity = .7] (1,1) to (3,1);
\draw[rounded corners = 2pt, line width = 8pt, green, line cap = round, opacity  = .7] (3,1) to (4,2);

\draw[thick, fill = black] (1,3) circle (3pt); 
\draw[thick, fill = black] (2,3) circle (3pt); 
\draw[thick, fill = black] (3,3) circle (3pt); 
\draw[thick, fill = black] (4,3) circle (3pt); 

\draw[thick, fill = black] (1,2) circle (3pt); 
\draw[thick, fill = black] (2,2) circle (3pt); 
\draw[thick, fill = black] (3,2) circle (3pt); 
\draw[thick, fill = black] (4,2) circle (3pt); 

\draw[thick, fill = black] (1,1) circle (3pt); 
\draw[thick, fill = black] (2,1) circle (3pt); 
\draw[thick, fill = black] (3,1) circle (3pt); 
\draw[thick, fill = black] (4,1) circle (3pt);

\end{scope}


\begin{scope}[shift={(9,0)}]

\draw[rounded corners = 10pt, fill = gray!40!, opacity =.2] (1-.4,3+.4) -- (4+.4,3+.4) -- (4+.4,1-.4) -- (1-.4,1-.4) -- cycle;

\draw[rounded corners = 2pt, line width = 10pt, red, line cap = round, opacity  = .5] (3,1) to (4,2);

\draw[rounded corners = 2pt, line width = 10pt, blue, line cap = round, opacity  = .5] (1,3) to (2,2);

\draw[rounded corners = 2pt, line width = 10pt, green, line cap = round, opacity  = .7] (3,3) to (2,1);

\node at (1.5,2.5) {$i$};
\node at (2.5,2) {$i$};
\node at (3.5,1.5) {$i$};

\draw[thick, fill= black] (1,3) circle (3pt); 
\draw[thick, fill= black] (2,3) circle (3pt); 
\draw[thick, fill= black] (3,3) circle (3pt); 
\draw[thick, fill= black] (4,3) circle (3pt); 

\draw[thick, fill= black] (1,2) circle (3pt); 
\draw[thick, fill= black] (2,2) circle (3pt); 
\draw[thick, fill= black] (3,2) circle (3pt); 
\draw[thick, fill= black] (4,2) circle (3pt); 

\draw[thick, fill= black] (1,1) circle (3pt); 
\draw[thick, fill= black] (2,1) circle (3pt); 
\draw[thick, fill= black] (3,1) circle (3pt); 
\draw[thick, fill= black] (4,1) circle (3pt); 

\node at (2.5,0) {$\underbrace{\hspace{3.5cm}}$};
\node at (2.5,-.5) {$b\in B$};

\end{scope}

\end{tikzpicture}

%% file: hypergraph-conflict.tex
\begin{tikzpicture}
 
\begin{scope}[shift = {(4,0)}]

\draw[thin, dotted] (.25,0) to (3.75,0);
\draw[thin, dotted] (.25,1) to (3.75,1);
\draw[thin, dotted] (.25,2) to (3.75,2);
\draw[thin, dotted] (.25,3) to (3.75,3);
\draw[thin, dotted] (.25,4.5) to (3.75,4.5);
\draw[thin, dotted] (.25,5.5) to (3.75,5.5);

\draw[rounded corners = 10pt, fill = gray!40!, opacity =.2] (1-.4,5.5+.4) -- (3+.4,5.5+.4) -- (3+.4,0-.4) -- (1-.4,0-.4) -- cycle;

\draw[rounded corners = 2pt, line width = 4pt, red, line cap = round, opacity = .3] (1,0) -- (3,1) -- (2,2);

\begin{scope}
\clip (1.5,2) rectangle (2.5,3.5);
\draw[rounded corners = 2pt, line width = 4pt, red, line cap = round, bend right = 20, opacity = .3] 
    (2,2) to (2,3) to (2,4);
\draw[rounded corners = 2pt, line width = 4pt, red, line cap = round, bend right = 70, opacity = .3] (2,2) to (2,3) to (2,4);
\end{scope}

\draw[rounded corners = 2pt, line width = 4pt, red, line cap = round, opacity = .3] (3,0) -- (1,1) -- (2,2);

\draw[rounded corners = 2pt, line width = 4pt, blue, line cap = round, opacity = .3] (1,0) -- (1,1);
\draw[rounded corners = 2pt, line width = 4pt, blue, line cap = round, opacity = .3] (3,0) -- (3,1);

\begin{scope}
  \clip (1.5,2) rectangle (2.5,3.5);
\draw[rounded corners = 2pt, line width = 4pt, blue, line cap = round, bend left = 20, opacity = .3] (2,2) to (2,3) to (2,4);
\draw[rounded corners = 2pt, line width = 4pt, blue, line cap = round, bend left = 70, opacity = .3] (2,2) to (2,3) to (2,4);
\end{scope}

\draw[rounded corners = 2pt, line width = 4pt, blue, line cap = round, bend left = 30, opacity = .3] (1,1) to (2,2);
\draw[rounded corners = 2pt, line width = 4pt, blue, line cap = round, bend right = 30, opacity = .3] (3,1) to (2,2);

\begin{scope}
\clip (1.5,4) rectangle (2.5,5.5);
\draw[rounded corners = 2pt, line width = 4pt, red, line cap = round, bend right = 20, opacity = .3] 
    (2,3.5) to (2,4.5) to (2,5.5);
\draw[rounded corners = 2pt, line width = 4pt, red, line cap = round, bend right = 70, opacity = .3] (2,3.5) to (2,4.5) to (2,5.5);
\end{scope}

\begin{scope}
\clip (1.5,4) rectangle (2.5,5.5);
\draw[rounded corners = 2pt, line width = 4pt, blue, line cap = round, bend left = 20, opacity = .3] (2,3.5) to (2,4.5) to (2,5.5);
\draw[rounded corners = 2pt, line width = 4pt, blue, line cap = round, bend left = 70, opacity = .3] (2,3.5) to (2,4.5) to (2,5.5);
\end{scope}
           
\draw[fill = black] (1,0) circle (3pt);
\draw[fill = black] (3,0) circle (3pt);
\draw[fill = black] (1,1) circle (3pt);
\draw[fill = black] (3,1) circle (3pt);
\draw[fill = black] (2,2) circle (3pt);
\draw[fill = black] (2,3) circle (3pt);
\draw[fill = black] (2,4.5) circle (3pt);
\draw[fill = black] (2,5.5) circle (3pt);

\node at (2,3.875) {\tiny{$\vdots$}};
\end{scope}


\begin{scope}[shift = {(4,0)}]
\draw[rounded corners = 2pt, line width = 1pt, red, line cap = round, opacity = 1, dashed] (1,0) -- (3,1) -- (2,2);

\begin{scope}
\clip (1.5,2) rectangle (2.5,3.5);
\draw[rounded corners = 2pt, line width = 1pt, red, line cap = round, bend right = 20, opacity = 1, dashed] (2,2) to (2,3) to (2,4);
\draw[rounded corners = 2pt, line width = 1pt, red, line cap = round, bend right = 70, opacity = 1] (2,2) to (2,3) to (2,4);
\end{scope}

\draw[rounded corners = 2pt, line width = 1pt, red, line cap = round, opacity = 1] (3,0) -- (1,1) -- (2,2);

\draw[rounded corners = 2pt, line width = 1pt, blue, line cap = round, opacity = 1] (1,0) -- (1,1);
\draw[rounded corners = 2pt, line width = 1pt, blue, line cap = round, opacity = 1, dashed] (3,0) -- (3,1);

\begin{scope}
 \clip (1.5,2) rectangle (2.5,3.5);
\draw[rounded corners = 2pt, line width = 1pt, blue, line cap = round, bend left = 20, opacity = 1, dashed] (2,2) to (2,3) to (2,4);
\draw[rounded corners = 2pt, line width = 1pt, blue, line cap = round, bend left = 70, opacity = 1] (2,2) to (2,3) to (2,4);
\end{scope}

\draw[rounded corners = 2pt, line width = 1pt, blue, line cap = round, bend left = 30, opacity = 1] (1,1) to (2,2);
\draw[rounded corners = 2pt, line width = 1pt, blue, line cap = round, bend right = 30, opacity = 1, dashed] (3,1) to (2,2);

\begin{scope}
 \clip (1.5,4) rectangle (2.5,5.5);
\draw[rounded corners = 2pt, line width = 1pt, red, line cap = round, bend right = 20, opacity = 1, dashed] (2,3.5) to (2,4.5) to (2,5.5);
\draw[rounded corners = 2pt, line width = 1pt, red, line cap = round, bend right = 70, opacity = 1] (2,3.5) to (2,4.5) to (2,5.5);
\end{scope}

\begin{scope}
\clip (1.5,4) rectangle (2.5,5.5);
\draw[rounded corners = 2pt, line width = 1pt, blue, line cap = round, bend left = 20, opacity = 1, dashed] (2,3.5) to (2,4.5) to (2,5.5);
\draw[rounded corners = 2pt, line width = 1pt, blue, line cap = round, bend left = 70, opacity = 1] (2,3.5) to (2,4.5) to (2,5.5);
\end{scope}
           
\draw[fill = black] (1,0) circle (3pt);
\draw[fill = black] (3,0) circle (3pt);
\draw[fill = black] (1,1) circle (3pt);
\draw[fill = black] (3,1) circle (3pt);
\draw[fill = black] (2,2) circle (3pt);
\draw[fill = black] (2,3) circle (3pt);
\draw[fill = black] (2,4.5) circle (3pt);
\draw[fill = black] (2,5.5) circle (3pt);

\node at (2,3.875) {\tiny{$\vdots$}};
\end{scope}

\end{tikzpicture}

%% file: hypergraph-conflict-D.tex
\begin{tikzpicture}
 
\begin{scope}[shift = {(4,0)}]

\draw[thin, dotted] (.25,0) to (3.75,0);
\draw[thin, dotted] (.25,1) to (3.75,1);
\draw[thin, dotted] (.25,2) to (3.75,2);
\draw[thin, dotted] (.25,3) to (3.75,3);
\draw[thin, dotted] (.25,4.5) to (3.75,4.5);
\draw[thin, dotted] (.25,5.5) to (3.75,5.5);

\draw[rounded corners = 10pt, fill = gray!40!, opacity =.2] (1-.4,5.5+.4) -- (3+.4,5.5+.4) -- (3+.4,0-.4) -- (1-.4,0-.4) -- cycle;

\draw[rounded corners = 2pt, line width = 4pt, red, line cap = round, opacity = .3] (1,0) -- (3,1) -- (2,2);

\begin{scope}
\clip (1.5,2) rectangle (2.5,3.5);
\draw[rounded corners = 2pt, line width = 4pt, red, line cap = round, bend right = 20, opacity = .3] 
    (2,2) to (2,3) to (2,4);
\draw[rounded corners = 2pt, line width = 4pt, red, line cap = round, bend right = 70, opacity = .3] (2,2) to (2,3) to (2,4);
\end{scope}

\draw[rounded corners = 2pt, line width = 4pt, red, line cap = round, opacity = .3] (3,0) -- (1,1) -- (2,2);

\draw[rounded corners = 2pt, line width = 4pt, blue, line cap = round, opacity = .3] (1,0) -- (1,1);
\draw[rounded corners = 2pt, line width = 4pt, blue, line cap = round, opacity = .3] (3,0) -- (3,1);

\begin{scope}
  \clip (1.5,2) rectangle (2.5,3.5);
\draw[rounded corners = 2pt, line width = 4pt, blue, line cap = round, bend left = 20, opacity = .3] (2,2) to (2,3) to (2,4);
\draw[rounded corners = 2pt, line width = 4pt, blue, line cap = round, bend left = 70, opacity = .3] (2,2) to (2,3) to (2,4);
\end{scope}

\draw[rounded corners = 2pt, line width = 4pt, blue, line cap = round, bend left = 30, opacity = .3] (1,1) to (2,2);
\draw[rounded corners = 2pt, line width = 4pt, blue, line cap = round, bend right = 30, opacity = .3] (3,1) to (2,2);

\begin{scope}
\clip (1.5,4) rectangle (2.5,5.5);
\draw[rounded corners = 2pt, line width = 4pt, red, line cap = round, bend right = 20, opacity = .3] 
    (2,3.5) to (2,4.5) to (2,5.5);
\draw[rounded corners = 2pt, line width = 4pt, red, line cap = round, bend right = 70, opacity = .3] (2,3.5) to (2,4.5) to (2,5.5);
\end{scope}

\begin{scope}
\clip (1.5,4) rectangle (2.5,5.5);
\draw[rounded corners = 2pt, line width = 4pt, blue, line cap = round, bend left = 20, opacity = .3] (2,3.5) to (2,4.5) to (2,5.5);
\draw[rounded corners = 2pt, line width = 4pt, blue, line cap = round, bend left = 70, opacity = .3] (2,3.5) to (2,4.5) to (2,5.5);
\end{scope}
           
\draw[fill = black] (1,0) circle (3pt);
\draw[fill = black] (3,0) circle (3pt);
\draw[fill = black] (1,1) circle (3pt);
\draw[fill = black] (3,1) circle (3pt);
\draw[fill = black] (2,2) circle (3pt);
\draw[fill = black] (2,3) circle (3pt);
\draw[fill = black] (2,4.5) circle (3pt);
\draw[fill = black] (2,5.5) circle (3pt);

\node at (2,3.875) {\tiny{$\vdots$}};
\end{scope}


\begin{scope}[shift = {(4,0)}]
\draw[rounded corners = 2pt, line width = 1pt, red, line cap = round, opacity = 1, dashed] (1,0) -- (3,1) -- (2,2);

\begin{scope}
\clip (1.5,2) rectangle (2.5,3.5);
\draw[rounded corners = 2pt, line width = 1pt, red, line cap = round, bend right = 20, opacity = 1, dashed] (2,2) to (2,3) to (2,4);
\draw[rounded corners = 2pt, line width = 1pt, red, line cap = round, bend right = 70, opacity = 1] (2,2) to (2,3) to (2,4);
\end{scope}

\draw[rounded corners = 2pt, line width = 1pt, red, line cap = round, opacity = 1] (3,0) -- (1,1) -- (2,2);

\draw[rounded corners = 2pt, line width = 1pt, blue, line cap = round, opacity = 1] (1,0) -- (1,1);
\draw[rounded corners = 2pt, line width = 1pt, blue, line cap = round, opacity = 1, dashed] (3,0) -- (3,1);

\begin{scope}
 \clip (1.5,2) rectangle (2.5,3.5);
\draw[rounded corners = 2pt, line width = 1pt, blue, line cap = round, bend left = 20, opacity = 1, dashed] (2,2) to (2,3) to (2,4);
\draw[rounded corners = 2pt, line width = 1pt, blue, line cap = round, bend left = 70, opacity = 1] (2,2) to (2,3) to (2,4);
\end{scope}

\draw[rounded corners = 2pt, line width = 1pt, blue, line cap = round, bend left = 30, opacity = 1] (1,1) to (2,2);
\draw[rounded corners = 2pt, line width = 1pt, blue, line cap = round, bend right = 30, opacity = 1, dashed] (3,1) to (2,2);

\begin{scope}
 \clip (1.5,4) rectangle (2.5,5.5);
\draw[rounded corners = 2pt, line width = 1pt, red, line cap = round, bend right = 20, opacity = 1, dashed] (2,3.5) to (2,4.5) to (2,5.5);
\draw[rounded corners = 2pt, line width = 1pt, red, line cap = round, bend right = 70, opacity = 1] (2,3.5) to (2,4.5) to (2,5.5);
\end{scope}

\begin{scope}
\clip (1.5,4) rectangle (2.5,5.5);
\draw[rounded corners = 2pt, line width = 1pt, blue, line cap = round, bend left = 20, opacity = 1, dashed] (2,3.5) to (2,4.5) to (2,5.5);
\draw[rounded corners = 2pt, line width = 1pt, blue, line cap = round, bend left = 70, opacity = 1] (2,3.5) to (2,4.5) to (2,5.5);
\end{scope}
           
\draw[fill = black] (1,0) circle (3pt);
\draw[fill = black] (3,0) circle (3pt);
\draw[fill = black] (1,1) circle (3pt);
\draw[fill = black] (3,1) circle (3pt);
\draw[fill = black] (2,2) circle (3pt);
\draw[fill = black] (2,3) circle (3pt);
\draw[fill = black] (2,4.5) circle (3pt);
\draw[fill = black] (2,5.5) circle (3pt);

\node at (2,3.875) {\tiny{$\vdots$}};
\end{scope}

\begin{scope}[shift = {(4+6,0)}]

\draw[thin, dotted] (.25,0) to (3.75,0);
\draw[thin, dotted] (.25,1) to (3.75,1);
\draw[thin, dotted] (.25,2) to (3.75,2);
\draw[thin, dotted] (.25,3) to (3.75,3);
\draw[thin, dotted] (.25,4.5) to (3.75,4.5);
\draw[thin, dotted] (.25,5.5) to (3.75,5.5);

\draw[rounded corners = 10pt, fill = gray!40!, opacity =.2] (1-.4,5.5+.4) -- (3+.4,5.5+.4) -- (3+.4,0-.4) -- (1-.4,0-.4) -- cycle;

\draw[rounded corners = 2pt, line width = 4pt, red, line cap = round, opacity = .3] (1,0) -- (3,1) -- (2,2);

\begin{scope}
\clip (1.5,2) rectangle (2.5,3.5);
\draw[rounded corners = 2pt, line width = 4pt, red, line cap = round, bend right = 20, opacity = .3] 
    (2,2) to (2,3) to (2,4);
\draw[rounded corners = 2pt, line width = 4pt, red, line cap = round, bend right = 70, opacity = .3] (2,2) to (2,3) to (2,4);
\end{scope}

\draw[rounded corners = 2pt, line width = 4pt, red, line cap = round, opacity = .3] (3,0) -- (1,1) -- (2,2);

\draw[rounded corners = 2pt, line width = 4pt, blue, line cap = round, opacity = .3] (1,0) -- (1,1);
\draw[rounded corners = 2pt, line width = 4pt, blue, line cap = round, opacity = .3] (3,0) -- (3,1);

\begin{scope}
  \clip (1.5,2) rectangle (2.5,3.5);
\draw[rounded corners = 2pt, line width = 4pt, blue, line cap = round, bend left = 20, opacity = .3] (2,2) to (2,3) to (2,4);
\draw[rounded corners = 2pt, line width = 4pt, blue, line cap = round, bend left = 70, opacity = .3] (2,2) to (2,3) to (2,4);
\end{scope}

\draw[rounded corners = 2pt, line width = 4pt, blue, line cap = round, bend left = 30, opacity = .3] (1,1) to (2,2);
\draw[rounded corners = 2pt, line width = 4pt, blue, line cap = round, bend right = 30, opacity = .3] (3,1) to (2,2);

\begin{scope}
\clip (1.5,4) rectangle (2.5,5.5);
\draw[rounded corners = 2pt, line width = 4pt, red, line cap = round, bend right = 20, opacity = .3] 
    (2,3.5) to (2,4.5) to (2,5.5);
\draw[rounded corners = 2pt, line width = 4pt, red, line cap = round, bend right = 70, opacity = .3] (2,3.5) to (2,4.5) to (2,5.5);
\end{scope}

\begin{scope}
\clip (1.5,4) rectangle (2.5,5.5);
\draw[rounded corners = 2pt, line width = 4pt, blue, line cap = round, bend left = 20, opacity = .3] (2,3.5) to (2,4.5) to (2,5.5);
\draw[rounded corners = 2pt, line width = 4pt, blue, line cap = round, bend left = 70, opacity = .3] (2,3.5) to (2,4.5) to (2,5.5);
\end{scope}
           
\draw[fill = black] (1,0) circle (3pt);
\draw[fill = black] (3,0) circle (3pt);
\draw[fill = black] (1,1) circle (3pt);
\draw[fill = black] (3,1) circle (3pt);
\draw[fill = black] (2,2) circle (3pt);
\draw[fill = black] (2,3) circle (3pt);
\draw[fill = black] (2,4.5) circle (3pt);
\draw[fill = black] (2,5.5) circle (3pt);

\node at (2,3.875) {\tiny{$\vdots$}};
\end{scope}


\begin{scope}[shift = {(4+6,0)}]
\draw[rounded corners = 2pt, line width = 1pt, red, line cap = round, opacity = 1, dashed] (1,0) -- (3,1) -- (2,2);

\begin{scope}
\clip (1.5,2) rectangle (2.5,3.5);
\draw[rounded corners = 2pt, line width = 1pt, red, line cap = round, bend right = 20, opacity = 1, dashed] (2,2) to (2,3) to (2,4);
\draw[rounded corners = 2pt, line width = 1pt, red, line cap = round, bend right = 70, opacity = 1] (2,2) to (2,3) to (2,4);
\end{scope}

\draw[rounded corners = 2pt, line width = 1pt, red, line cap = round, opacity = 1] (3,0) -- (1,1) -- (2,2);

\draw[rounded corners = 2pt, line width = 1pt, blue, line cap = round, opacity = 1] (1,0) -- (1,1);
\draw[rounded corners = 2pt, line width = 1pt, blue, line cap = round, opacity = 1, dashed] (3,0) -- (3,1);

\begin{scope}
 \clip (1.5,2) rectangle (2.5,3.5);
\draw[rounded corners = 2pt, line width = 1pt, blue, line cap = round, bend left = 20, opacity = 1, dashed] (2,2) to (2,3) to (2,4);
\draw[rounded corners = 2pt, line width = 1pt, blue, line cap = round, bend left = 70, opacity = 1] (2,2) to (2,3) to (2,4);
\end{scope}

\draw[rounded corners = 2pt, line width = 1pt, blue, line cap = round, bend left = 30, opacity = 1] (1,1) to (2,2);
\draw[rounded corners = 2pt, line width = 1pt, blue, line cap = round, bend right = 30, opacity = 1, dashed] (3,1) to (2,2);

\begin{scope}
 \clip (1.5,4) rectangle (2.5,5.5);
\draw[rounded corners = 2pt, line width = 1pt, red, line cap = round, bend right = 20, opacity = 1, dashed] (2,3.5) to (2,4.5) to (2,5.5);
\draw[rounded corners = 2pt, line width = 1pt, red, line cap = round, bend right = 70, opacity = 1] (2,3.5) to (2,4.5) to (2,5.5);
\end{scope}

\begin{scope}
\clip (1.5,4) rectangle (2.5,5.5);
\draw[rounded corners = 2pt, line width = 1pt, blue, line cap = round, bend left = 20, opacity = 1, dashed] (2,3.5) to (2,4.5) to (2,5.5);
\draw[rounded corners = 2pt, line width = 1pt, blue, line cap = round, bend left = 70, opacity = 1] (2,3.5) to (2,4.5) to (2,5.5);
\end{scope}
           
\draw[fill = black] (1,0) circle (3pt);
\draw[fill = black] (3,0) circle (3pt);
\draw[fill = black] (1,1) circle (3pt);
\draw[fill = black] (3,1) circle (3pt);
\draw[fill = black] (2,2) circle (3pt);
\draw[fill = black] (2,3) circle (3pt);
\draw[fill = black] (2,4.5) circle (3pt);
\draw[fill = black] (2,5.5) circle (3pt);

\node at (2,3.875) {\tiny{$\vdots$}};
\end{scope}

\draw[very thick, ->, bend left = 30] (6,-.65) to (5.5,.65);
\draw[very thick, ->, bend right = 30] (6,-.65) to (6.5,.65);

\draw[very thick, ->, bend left = 70] (5.1,-1.5) to (4.9,.6);
\draw[very thick, ->, bend right = 70] (6.9,-1.5) to (7.1,.6);

\node[draw, fill=gray!10, rounded corners, inner sep=3pt] at (6,-.9) {\footnotesize{Colors in $N_2$}};

\node[draw, fill=gray!10, rounded corners, inner sep=3pt] at (6,-1.5) {\footnotesize{Colors in $N_1$}};


\draw[very thick, ->, bend left = 30] (6+6,-.65) to (5.5+6,.65);
\draw[very thick, ->, bend right = 30] (6+6,-.65) to (6.5+6,.65);

\draw[very thick, ->, bend left = 70] (5.00+6,-.9) to (4.9+6,.6);
\draw[very thick, ->, bend right = 70] (7+6,-.9) to (7.1+6,.6);

\node[draw, fill=gray!10, rounded corners, inner sep=3pt] at (6+6,-.9) {\footnotesize{Colors in $N_2$}};


\end{tikzpicture}